\documentclass[12pt,british]{article}
\usepackage{amsmath,amsfonts,amssymb,amscd,amsthm}
\usepackage{babel}
\usepackage{enumerate}

\usepackage[matrix,arrow]{xy}

\makeatletter
\renewcommand{\theenumi}{\alph{enumi}}
\renewcommand{\theenumii}{\roman{enumii}}

\renewcommand{\p@enumi}{}
\renewcommand{\p@enumii}{}
\renewcommand{\p@enumiii}{}
\makeatother

\renewcommand{\epsilon}{\ensuremath{\varepsilon}}
\renewcommand{\to}{\ensuremath{\longrightarrow}}
\newcommand{\R}{\ensuremath{\mathbb R}}
\newcommand{\rp}{\ensuremath{{\mathbb R}P^2}}

\newcommand{\N}{\ensuremath{\mathbb N}}

\newcommand{\Z}{\ensuremath{\mathbb Z}}

\newcommand{\St}[1][n]{\ensuremath{\mathbb S}^{#1}}

\newcommand{\quat}[1][8]{\ensuremath{\mathcal{Q}_{#1}}}

\makeatletter
\def\@enum@{\list{\csname label\@enumctr\endcsname}%
           {\usecounter{\@enumctr}\def\makelabel##1{
\normalfont\ignorespaces\emph{{##1}~}}
\setlength{\labelsep}{3pt}
\setlength{\parsep}{0pt}
\setlength{\itemsep}{0pt}
\setlength{\leftmargin}{0pt}
\setlength{\labelwidth}{0pt}
\setlength{\listparindent}{\parindent}%
\setlength{\itemsep}{0pt}
\setlength{\itemindent}{0pt}
\setlength{\topsep}{3pt plus 1pt minus 1 pt}}}
\renewcommand\theenumi{\@alph\c@enumi}
\renewcommand\theenumii{\@alph\c@enumii}
\renewcommand\theenumiii{\@alph\c@enumiii}
\renewcommand\theenumiv{\@alph\c@enumiv}

\makeatother

\newcommand{\map}[4][\to]{\ensuremath{{#2} \colon\thinspace {#3} #1 {#4}}}

\DeclareRobustCommand*{\up}[1]{\textsuperscript{#1}}

\newcommand{\brak}[1]{\ensuremath{\left\{ #1 \right\}}}
\newcommand{\ang}[1]{\ensuremath{\left\langle #1\right\rangle}}
\newcommand{\set}[2]{\ensuremath{\brak{#1 \,\mid\, #2}}}

\newcommand{\setang}[2]{\ensuremath{\ang{#1 \,\mid\, #2}}}
\newcommand{\setangr}[2]{\ensuremath{\ang{#1 \,\left\lvert \, #2 \right.}}}
\newcommand{\setangl}[2]{\ensuremath{\ang{\left. #1 \,\right\rvert \, #2}}}

\newcommand{\setr}[2]{\ensuremath{\brak{#1 \,\left\lvert \, #2 \right.}}}
\newcommand{\setl}[2]{\ensuremath{\brak{\left. #1 \,\right\rvert \, #2}}}

\newtheoremstyle{theoremm}{}{}{\itshape}{}{\scshape}{.}{ }{}

\theoremstyle{theoremm}
\newtheorem{thm}{Theorem}
\newtheorem{lem}[thm]{Lemma}
\newtheorem{prop}[thm]{Proposition}
\newtheorem{cor}[thm]{Corollary}

\newtheoremstyle{remarkk}{}{}{}{}{\scshape}{.}{ }{}

\theoremstyle{remarkk}

\newtheorem{rem}[thm]{Remark}

\newcommand{\relem}[1]{Lemma~\protect\ref{lem:#1}}
\newcommand{\repr}[1]{Proposition~\protect\ref{prop:#1}}
\newcommand{\reco}[1]{Corollary~\protect\ref{cor:#1}}
\newcommand{\resec}[1]{Section~\protect\ref{sec:#1}}
\newcommand{\req}[1]{equation~(\protect\ref{eq:#1})}
\newcommand{\reqref}[1]{(\protect\ref{eq:#1})}
\newcommand{\rerem}[1]{Remark~\protect\ref{rem:#1}}

\begin{document}

\title{The Borsuk-Ulam theorem for maps into a surface}

\author{DACIBERG~LIMA~GON\c{C}ALVES\\
Departamento de Matem\'atica - IME-USP,\\
Caixa Postal~66281~-~Ag.~Cidade de S\~ao Paulo,\\ 
CEP:~ 05314-970 - S\~ao Paulo - SP - Brazil.\\
e-mail:~\texttt{dlgoncal@ime.usp.br}\vspace*{4mm}\\
JOHN~GUASCHI\\
Laboratoire de Math\'ematiques Nicolas Oresme UMR CNRS~\textup{6139},\\
Universit\'e de Caen BP 5186,
14032 Caen Cedex, France,\\
and\\
Instituto de Matem\'aticas, UNAM,
Le\'on \#2, altos, col.\ centro,\\
Oaxaca de Ju\'arez, Oaxaca,
C.P. 68000, Mexico.\\
e-mail:~\texttt{guaschi@math.unicaen.fr}}

\date{24th July 2009: revised 5th December 2009}

\maketitle

\begin{abstract}
\noindent


Let $(X, \tau, S)$ be a triple, where $S$ is a compact, connected surface without boundary, and $\tau$ is a free cellular involution on a $CW$-complex $X$. The triple $(X, \tau, S)$ is said to satisfy the \emph{Borsuk-Ulam property} if for every continuous map $\map{f}{X}S$, there exists a point $x\in X$ satisfying $f(\tau(x))=f(x)$. In this paper, we formulate this property in terms of a relation in the $2$-string braid group $B_{2}(S)$ of $S$. If $X$ is a compact, connected surface without boundary, we use this criterion to classify all triples $(X, \tau, S)$ for which the Borsuk-Ulam property holds.  We also consider various cases where $X$ is not necessarily a surface without boundary, but has the property that $\pi_1(X/\tau)$ is isomorphic to the fundamental group of such a surface. 
If $S$ is different from the $2$-sphere $\St[2]$ and the real projective plane $\rp$, then we show that the Borsuk-Ulam property does not hold for $(X, \tau, S)$ unless either $\pi_1(X/\tau)\cong \pi_{1}(\rp)$, or $\pi_1(X/\tau)$ is isomorphic to the fundamental group of a compact, connected non-orientable surface of genus $2$ or $3$ and $S$ is non-orientable. In the latter case, the veracity of the Borsuk-Ulam property depends further on the choice of involution $\tau$; we give a necessary and sufficient condition for it to hold in terms of the surjective homomorphism $\pi_1(X/\tau)\to \Z_{2}$ induced by the double covering $X\to X/\tau$. The cases $S=\St[2],\rp$ are treated separately.
\end{abstract}

\section{Introduction}\label{sec:intro}

St.~Ulam conjectured that \emph{if $\map{f}{\St}{\R^n}$ is a continuous map then there exists a point $p\in \St$ such that $f(p)=f(-p)$, where $-p$ is the antipodal point of $p$}~\cite[footnote, page 178]{Bo}. The conjecture was solved in~1933 by K.~Borsuk~\cite[S\"atz II]{Bo}. There was another result in Borsuk's paper, S\"atz~III, which is indeed equivalent to S\"atz~II (see~\cite[Section 2, Theorem 2.1.1]{Mat}).
It turned out that S\"atz~III had been proved three years before by L.~Lusternik and  L.~Schnirelmann~\cite{LS} (see also \cite[footnote, page 190]{Bo}). This was the beginning of the history of what we shall refer to as the \emph{Borsuk-Ulam property} or \emph{Borsuk-Ulam type theorem}. We say that \emph{the triple $(X, \tau, S)$ has the Borsuk-Ulam property 
if for every continuous map $\map{f}{X}{S}$, there is a point $x\in X$ such that $f(x)=f(\tau(x))$}.  In the past seventy years, the original statement has been greatly generalised in many directions, and has also been studied in other natural contexts. The contributions are numerous, and we do not intend to present here a detailed description of the development of the subject. One may consult~\cite{Mat} for some applications of the Borsuk-Ulam theorem.

In this Introduction, we concentrate on a particular direction that is more closely related to the type of Borsuk-Ulam problem relevant to the main theme of this paper. In~\cite{Go}, a Borsuk-Ulam type theorem for maps from compact surfaces without boundary with free involutions  into $\R^2$ was studied. 
An important feature which appears in these results of that paper is that the validity of the theorem depends upon the choice of involution. This phenomenon did not and could not show up in the case where the domain is the $2$-sphere $\St[2]$ since up to conjugation there is only one free involution on $\St[2]$. In a similar vein, the Borsuk-Ulam property was also analysed for triples for which the domain is a $3$-space form in~\cite{GNS}, and also for Seifert manifolds in~\cite{GHZ,BGHZ}. The study of these papers leads us to formulate a general problem which consists in finding the maximal value $n$ for which the Borsuk-Ulam property is true for triples $(X, \tau, \R^n)$, 
where $X$ is a given finite-dimensional $CW$-complex $X$ equipped with a free involution $\tau$. In this paper, we choose a direction closer to that of~\cite{Go} which is the investigation of maps from a space whose fundamental group is that of a surface, into a compact, connected surface $S$ without boundary. Within this framework, Proposition~\ref{gen} will enable us to formulate the veracity of the Borsuk-Ulam property in terms of a commutative diagram of the $2$-string braid group $B_{2}(S)$ of $S$. We shall then apply algebraic properties of $B_{2}(S)$ to help us to decide whether   the Borsuk-Ulam property holds in our setting in all cases.

Throughout this paper, $S$ will always denote a compact, connected surface without boundary, $S_g$ will be a compact, orientable surface of genus $g\geq 0$ without boundary, and $N_l$ will be a compact, non-orientable surface of genus $l\geq 1$ without boundary. We consider triples $(X, \tau, S)$, where $X$ is a $CW$-complex and $\tau$ is a cellular free involution.
The following statements summarise our main results. 

\begin{cor}\label{2dcw}\mbox{}
 \begin{enumerate}[$($a$)$]
\item\label{it:2dcw3} If $X$ is a $CW$-complex equipped with a cellular free involution $\tau$, the triple $(X, \tau, \St[2])$ satisfies the Borsuk-Ulam property if and only if the triple $(X, \tau, \R^3)$ satisfies the Borsuk-Ulam property. 
\item\label{it:2dcw3b} If $X$ is a $2$-dimensional $CW$-complex, the triple $(X, \tau, \St[2])$ does not satisfy the Borsuk-Ulam property for any cellular free involution $\tau$.
\item\label{it:2dcw3c} The triple $(\St[3], \tau, \St[2])$ satisfies the Borsuk-Ulam property for the unique cellular free involution (up to conjugacy) $\tau$ on $\St[3]$.
\item\label{it:2dcw3d} The triple $(\mathbb{R} P^3, \tau, \St[2])$ does not satisfy the Borsuk-Ulam property for the unique cellular free involution (up to conjugacy) $\tau$
on $\mathbb{R} P^3$.
\end{enumerate}
\end{cor}

If the target is the projective plane $\rp$ we have:
\begin{thm}\label{th:rp2case}
Let $X$ be a $CW$-complex equipped with a cellular free involution $\tau$ of dimension less than or equal to three, and suppose that $\pi_1(X)$ is isomorphic to the fundamental group of a compact surface without boundary. Then the Borsuk-Ulam property holds for the triple $(X, \tau, \rp)$ if and 
only if $X$ is simply connected. In particular, if $X$ is a compact surface without boundary, then the Borsuk-Ulam property holds for the triple $(X, \tau, \rp )$ 
if and only if $X$ is the $2$-sphere.
\end{thm}

These two results thus treat the cases where $S=\St[2]$ or $\rp$. From now on, assume that $S$ is different from $\St[2]$ and $\rp$, that $X$ is a finite-dimensional $CW$-complex, equipped with a cellular free involution $\tau$, and that $\pi_1(X/\tau)$  is either finite or is isomorphic to the fundamental group of a compact surface without boundary. The condition that $\pi_1(X/\tau)$ is finite is of course equivalent to saying that $\pi_1(X)$ is finite.

\begin{rem}
If the above space $X$ is a finite-dimensional $CW$-complex that is a $K(\pi, 1)$, the
hypothesis that $\pi_1(X/\tau)$ is isomorphic to the fundamental group of a compact surface without boundary is equivalent to saying 
that $\pi_1(X)$ is isomorphic to the fundamental group of a compact surface without boundary. To see this, observe that $X/\tau$ is also a $K(\pi, 1)$ and 
a finite-dimensional $CW$-complex. Therefore the group $\pi_1(X/\tau)$ is torsion free and is the middle group of the short exact sequence
$1 \to \pi_1(X) \to  \pi_1(X/\tau) \to \Z_2 \to 1$. 
Since $\pi_1(X)$ is a surface group and of finite index in $\pi_1(X/\tau)$, it follows that
that $\pi_1(X/\tau)$ is also a surface group. Indeed, from~\cite[Proposition 10.2, Section VIII]{Br}, $\pi_1(X/\tau)$  is a duality group, and  has the same duality module $\Z$ as $\pi_1(X)$. So $\pi_1(X/\tau)$  is a Poincaré duality group over $\Z$.  But every $PD^2$ group over $\Z$ is the fundamental group of a surface as result of~\cite{EL,EM}.
\end{rem}

In the case that $\pi_1(X)$ is finite, we obtain the following result.
\begin{prop}\label{fin}     
Let $X$ be a  $CW$-complex equipped with a cellular free involution $\tau$, and let $S$ be a compact, connected surface without boundary and different from $\rp$ and $\St[2]$. If $\pi_1(X)$ is finite then the Borsuk-Ulam property holds for the triple $(X,\tau, S)$
\end{prop}

Now suppose that $\pi_1(X/\tau)$ is isomorphic to the fundamental group of a compact surface without boundary. There are four basic cases according to whether $S$ is orientable or non-orientable, and to whether $\pi_{1}(X/\tau)$ is isomorphic to the fundamental group of an orientable or a non-orientable surface without boundary. In Section~\ref{sec:rootspns}, we first consider the case where $S$ is non-orientable. The following theorem pertains to the first subcase where $\pi_{1}(X/\tau)$ is isomorphic to the fundamental group of an orientable surface without boundary.

\begin{thm}\label{orien} 
Let $X$ be a  finite-dimensional $CW$-complex equipped with a  cellular free involution $\tau$, and let $S$ be a compact, connected non-orientable surface without boundary and different from $\rp$. If $\pi_1(X/\tau)$ is isomorphic to the fundamental group of a compact, connected orientable surface without boundary then the Borsuk-Ulam property does not hold for the triple $(X,\tau, S)$. 
\end{thm}

For the second subcase where $\pi_{1}(X/\tau)$ is isomorphic to the fundamental group of a non-orientable surface without boundary, we have:

\begin{thm}\label{th:rootspns}
Let $X$ be a finite-dimensional $CW$-complex equipped with a  cellular free involution $\tau$, and let $S$ be a compact, connected non-orientable surface without boundary different from $\rp$. Suppose that $\pi_1(X/\tau)$ is isomorphic to the fundamental group of a compact, connected non-orientable surface without boundary. Then the Borsuk-Ulam property holds for the triple $(X,\tau,S)$ if and only if $\pi_1(X)=\brak{1}$.
\end{thm}

In Section~\ref{sec:ornotS2}, we study the second case, where $S$ is orientable. If $\pi_{1}(X/\tau)$ is isomorphic to the fundamental group of an orientable surface without boundary, we have:

\begin{thm}\label{orien-orien}
Let $X$ be a finite-dimensional $CW$-complex equipped with a  cellular free involution $\tau$, and let $g>0$. If $S=S_g$, and if $\pi_1(X/\tau)$ is isomorphic to the fundamental group of a compact, connected orientable surface without boundary then the Borsuk-Ulam property does not hold for the triple $(X,\tau, S)$.
\end{thm}

The remainder of Section~\ref{sec:ornotS2} is devoted to the study of the subcase where $\pi_{1}(X/\tau)$ is isomorphic to the fundamental group of the non-orientable surface $N_{l}$ without boundary and  $S=S_g$, where $g\geq 1$. Our analysis divides into four subcases:
\begin{enumerate}[(1)]
 \item\label{it:casea} $l=1$.
\item\label{it:caseb} $l \geq 4$.
\item\label{it:cased} $l=2$.
\item\label{it:casee} $l=3$.
\end{enumerate}

For subcase~(\ref{it:casea}) we have:

\begin{prop}\label{case1} 
Let $X$ be a finite-dimensional $CW$-complex equipped with a cellular free involution $\tau$, and let $g\geq 1$. If $\pi_1(X/\tau)$ is isomorphic to the fundamental group of  the projective plane $\rp$, then the Borsuk-Ulam property holds for $(X,\tau,S_{g})$. 
\end{prop}

For subcase~(\ref{it:caseb}) we have:
\begin{prop}\label{case2}  
Let $X$ be a finite-dimensional $CW$-complex equipped with a cellular free involution $\tau$, let $l\geq 4$, and let $g\geq 1$. If $\pi_1(X/\tau)$ is isomorphic to the fundamental group of the non-orientable surface $N_{l}$ then the Borsuk-Ulam property does not hold for $(X,\tau, S_{g})$.
\end{prop}


To describe the results in the remaining two subcases, we first need to introduce some notation and terminology. Let $(X, \tau, S)$ be a triple, where $\tau$ is a cellular free involution on $X$ and $S$ is a compact surface without boundary, and let
\begin{equation*}
\map{\theta_{\tau}}{\pi_1(X/\tau)}{\Z_2}
\end{equation*}
be the surjective homomorphism defined by the double covering $X \to X/\tau$. For subcases~(\ref{it:cased})  and~(\ref{it:casee}), the veracity of the Borsuk-Ulam property depends on the choice of the free involution $\tau$. As we shall see in Proposition~\ref{gen}, the relevant information concerning $\tau$ is encoded in $\theta_{\tau}$. The study of the possible $\theta_{\tau}$ may be simplified by considering the following equivalence relation (see also the end of Section~\ref{sec:GeBo}). Let $G$ be a group, and consider the set of   
elements of $\operatorname{Hom}(G, \Z_2)$ that are surjective homomorphisms (or equivalently the 
elements that are not the null homomorphism). Two surjective homomorphisms $\phi_1,\phi_2\in  \operatorname{Hom}(G,\Z_2)$ are said to be \emph{equivalent} if there is an isomorphism $\map{\varphi}{G}{G}$ such that $\phi_1\circ \varphi=\phi_2$. Taking $G=\pi_{1}(X/\tau)$, and using the results given in the Appendix, we shall see that many algebraic questions will depend only on the equivalence classes of this relation. This will help to reduce the number of cases to be analysed.

For subcase~(\ref{it:cased}), where $l=2$, we have: 

\begin{prop}\label{case4} 
Let $X$ be a finite-dimensional $CW$-complex equipped with a cellular free involution $\tau$, and let $g\geq 1$. Consider the presentation 
$\setang{\alpha, \beta}{\alpha\beta\alpha\beta^{-1}}$ of the fundamental group of the Klein bottle $K$. If  $\pi_1(X/\tau)$ is isomorphic to $\pi_{1}(K)$ then the 
Borsuk-Ulam property holds for the triple $(X,\tau,S_{g})$ if and only if  $\theta_{\tau}(\alpha)=\overline{1}$. 
\end{prop}

For subcase~(\ref{it:casee}), where $l=3$, we have:

\begin{thm}\label{case5}
Let $X$ be a finite-dimensional $CW$-complex equipped with a cellular free involution $\tau$, and suppose that $\pi_{1}(X/\tau)$ is isomorphic to $\pi_{1}(N_{3})$. Consider the presentation $\setang{v,a_{1},a_{2}}{v^2\cdot [a_{1},a_{2}]}$ of the fundamental group of $N_{3}$. Then the Borsuk-Ulam property holds for the triple $(X,\tau, S_g)$ if and only if $\theta_{\tau}$ is equivalent to the homomorphism $\map{\theta}{\pi_{1}(N_{3})}{\Z_{2}}$ given by $\theta(v)=\theta(a_1)=\overline{1}$ and $\theta(a_2)=\overline{0}$. 
\end{thm}


For subcases~(\ref{it:cased}) and~(\ref{it:casee}), observe that as a result of the relations of the given presentation of $\pi_1(N_2)$ (resp.\ $\pi_1(N_3)$), any map $\map{\theta}{J}{\Z_2}$ satisfying the conditions of Proposition~\ref{case4} (resp.\ Theorem~\ref{case5}) extends to a homomorphism, where $J$ is the set of generators of $\pi_1(N_2)$ (resp.\ $\pi_1(N_3)$). Therefore there is a double covering 
which corresponds to the kernel of $\theta$, and consequently the cases in question may be realised by some pair $(X, \tau)$ for some cellular free involution $\tau$.

Theorems~\ref{orien},~\ref{th:rootspns},~\ref{orien-orien} and~\ref{case5}, and Propositions~\ref{case1},~\ref{case2} and~\ref{case4} may be summarised as follows.

\begin{thm}
Let $X$ be a finite-dimensional $CW$-complex equipped with a cellular free involution $\tau$. Suppose that $S\neq \St[2],\rp$. Then the Borsuk-Ulam property holds for $(X,\tau,S)$ if and only if one of the following holds:
\begin{enumerate}[(a)]
\item $\pi_{1}(X/\tau)\cong \pi_{1}(\rp)$.
\item $S$ is orientable, and either
\begin{enumerate}[(i)]
\item $\pi_{1}(X/\tau)\cong \pi_1(N_{2})$, and $\theta_{\tau}(\alpha)=\overline{1}$ for the presentation of $N_{2}$ given in Proposition~\ref{case4}.
\item $\pi_{1}(X/\tau)\cong \pi_1(N_{3})$, and $\theta_{\tau}$ is equivalent to the homomorphism $\map{\theta}{\pi_{1}(N_{3})}{\Z_{2}}$ given by $\theta(v)=\theta(a_1)=\overline{1}$ and $\theta(a_2)=\overline{0}$ for the presentation of $N_{3}$ given in Theorem~\ref{case5}.
\end{enumerate}
\end{enumerate}
\end{thm}

This paper is organised as follows. In Section~\ref{sec:GeBo}, we recall some general definitions, and state and prove Proposition~\ref{gen} which highlights the relation between the short exact sequence $1\to \pi_{1}(X) \to \pi_{1}(X/\tau)\to \Z_{2}\to 1$, and the short exact sequence $1\to P_{2}(S)\to B_{2}(S)\to \Z_{2}\to 1$ of the pure and full $2$-string braid groups of $S$. This proposition will play a vital r\^ole in much of the paper. Part~(\ref{it:propcase2}) of Proposition~\ref{gen} brings to light two special cases where $S=\St[2]$ or $S=\rp$. The case $S=\St[2]$ will be treated in Corollary~\ref{2dcw}. In Section~\ref{sec:rp2}, we deal with the case $S=\rp$, and prove Theorem~\ref{th:rp2case}. In Section~\ref{sec:rootspns}, we study the case where $S$ is a compact, non-orientable surface without boundary different from $\rp$, and prove Theorem~\ref{th:rootspns}. Finally, in Section~\ref{sec:ornotS2}, we analyse the case where $S$ is a compact, orientable surface without boundary different from $\St[2]$, and prove Theorems~\ref{orien-orien} and~\ref{case5} and Propositions~\ref{case1}--\ref{case4}. The proof of Theorem~\ref{case5} relies on a long and somewhat delicate argument using the lower central series of $P_{2}(S)$. 

\subsection*{Acknowledgements}

This work took place during the visit of the second author to the
Departmento de Mate\-m\'atica do IME~--~Universidade de S\~ao Paulo during
the periods 31\up{st}~October~--~10\up{th}~November~2008 and 20\up{th}~May~--~3\up{rd}~June 2009, and of the visit of the first author to the Laboratoire de Math\'ematiques Nicolas Oresme, Universit\'e de Caen during the period 21\up{st}~November~--~21\up{st}~December~2008. This work was supported by the international Cooperation USP/Cofecub project n\up{o} 105/06, by the CNRS/CNPq project n\up{o}~21119 and by the~ANR project TheoGar n\up{o} ANR-08-BLAN-0269-02. The writing of part of this paper took place while  the second author was at the Instituto de Matem\'aticas, UNAM Oaxaca, Mexico. He would like to thank the CNRS for having granted him a `d\'el\'egation' during this period, CONACYT for partial financial support through its programme `Estancias postdoctorales y sab\'aticas vinculadas al fortalecimiento de la calidad del posgrado nacional', and the Instituto de Matem\'aticas for its hospitality and excellent working atmosphere.

We would like to thank the referee for a careful reading of this paper, and for many suggestions that substantially improved the previous version. In particular, we point out his/her suggestion of the terminology  `Borsuk-Ulam property'. We also wish to thank Boju Jiang for his detailed comments on the paper.

\section{Generalities}\label{sec:GeBo}

Let $S$ be a compact surface without boundary, and let $G$ be a finite group that acts freely on a topological space $X$. If $\map{f}{X}{S}$ is a continuous map, we say that an orbit of the action is \emph{singular with respect to $f$} if the restriction of $f$ to the orbit is non injective. In particular, if $G=\Z_2$, a singular orbit is an orbit that is sent to a point by $f$. We study here the existence of singular orbits in the case where the group $G$ is $\Z_2$. The case where $G$ is an arbitrary finite cyclic group will be considered elsewhere.

The existence of a free action of $\Z_2$ on $X$ is equivalent to that of a fixed-point free involution
$\map{\tau}{X}{X}$. Let $(X, \tau, S)$ be a triple, where $\tau$ is a free involution on $X$, and let $\map{\theta_{\tau}}{\pi_1(X/\tau)}{\Z_2}$ be the homomorphism defined by the double 
covering $X \to X/\tau$. Recall that $F_2(S)=\setl{(x,y)\in S\times S}{x\neq y}$ is the $2$-point configuration space of $S$, $D_2(S)$ is the orbit space of $F_2(S)$ by the free $\Z_2$-action $\map{\tau_{S}}{F_2(S)}{F_2(S)}$, where $\tau_S(x,y)=(y,x)$, and $P_2(S)=\pi_1(F_2(S))$ and $B_2(S)=\pi_1(D_2(S))$ are the pure and full $2$-string braid groups respectively of $S$~\cite{FaN}. Let $\map{\pi}{B_{2}(S)}{\Z_{2}}$ denote the surjective homomorphism that to a $2$-braid of $S$ associates its permutation, and let $\map{p}{X}{X/\tau}$ denote the quotient map.

The following result will play a key r\^ole in the rest of the paper.

\begin{prop}\label{gen}
Let $X$ be  a $CW$-complex  equipped with a cellular free involution $\tau$, and let $S$ be a compact, connected surface without boundary. Suppose that the Borsuk-Ulam  property does not hold for the triple $(X, \tau, S)$. Then there exists a homomorphism $\map{\phi}{\pi_1(X/\tau)}{B_2(S)}$ that makes the following diagram commute:
\begin{equation}\label{eq:basic}
\begin{xy}*!C\xybox{%
\xymatrix{%
\pi_1(X/\tau) \ar@{-->}[rr]^{\phi} \ar[rdd]_{\theta_{\tau}} &  &
\ar[dld]^-{\pi}  B_2(S)  \\
&& \\
 & \Z_{2} &
 }}
\end{xy}
\end{equation}
Conversely, if such a factorisation $\phi$ exists then the Borsuk-Ulam property does not hold in the following cases:
\begin{enumerate}[(a)]
\item\label{it:propcase1} the space $X$ is a $CW$-complex of dimension less than or equal to two.
\item\label{it:propcase2} $S$ is a compact, connected surface without boundary different from $\St[2]$ and $\rp$.
\item\label{it:propcase3} $S$ is the projective plane and $X$ is a CW-complex of dimension less than or equal to three. 
\end{enumerate}
\end{prop}

\begin{rem}
So if $X$ and $S$ are as in the first line of Proposition~\ref{gen}, and if further $S\neq \St[2],\rp$ then the Borsuk-Ulam property does not hold for the triple $(X,\tau,S)$ if and only if there exists a homomorphism $\map{\phi}{\pi_1(X/\tau)}{B_2(S)}$ that makes the diagram~(\ref{eq:basic}) commute.
\end{rem}

\begin{proof}[Proof of Proposition~\ref{gen}.]
Suppose first that the Borsuk-Ulam property does not hold for the triple $(X, \tau, S)$. Then there exists a map 
$\map{f}{X}{S}$ such that $f(x)\neq f(\tau(x))$ for all $x\in X$. Define the map $\map{\widehat{f}}{X}{F_{2}(S)}$ by $\widehat{f}(x)=( f(x), f(\tau(x)))$. Note that $\widehat{f}$ is $\Z_2$-equivariant with respect to the actions on $X$ and $F_{2}(S)$ given respectively by $\tau$ and $\tau_{S}$, and so induces a map $\map{\widetilde{f}}{X/\tau}{D_{2}(S)}$ of the corresponding quotient spaces defined by $\widetilde{f}(y)=\brak{f(x),f(\tau(x))}$, where $x\in p^{-1}(\brak{y})$. On the level of fundamental groups, we obtain the following commutative diagram of short exact sequences:
\begin{equation*}
\xymatrix{ 
1  \ar[r]  & \pi_1(X) \ar[r]^-{p_{\#}} \ar[d]_{\widehat{f}_{\#}} & \pi_{1}(X/\tau) 
\ar[r]^-{\theta_{\tau}} \ar[d]_{\widetilde{f}_{\#}} &  \Z_2  \ar[r] \ar[d]^{\rho} & 1\\
1 \ar[r] & P_2(S) \ar[r] & B_{2}(S) \ar[r]^-{\pi} &  \Z_2 \ar[r] & 1,}
\end{equation*}
where $\widehat{f}_{\#},\widetilde{f}_{\#}$ are the homomorphisms induced by $\widehat{f},\widetilde{f}$ respectively, and $\map{\rho}{\Z_{2}}{\Z_{2}}$ is the homomorphism induced on the quotients. We claim that $\rho$ is injective. To see this, let $\gamma\in \ker{\rho}$, let $x_{0}\in X/\tau$ be a basepoint, let $\widetilde{x_{0}}\in X$ be a lift of $x_{0}$, and let $c$ be a loop in $X/\tau$ based at $x_{0}$ such that $\theta_{\tau}(\ang{c})=\gamma$. Let $\widetilde{c}$ be the lift of $c$ based at $\widetilde{x_{0}}$. Thus $\widetilde{c}$ is an arc from $\widetilde{x_{0}}$ to a point of $\brak{\widetilde{x_{0}}, \tau(\widetilde{x_{0}})}$. We have that $\pi\circ \widetilde{f}_{\#}(\ang{c})=\rho\circ \theta_{\tau}(\ang{c})=\overline{0}$, so $\widetilde{f}_{\#}(\ang{c})\in \ker{\pi}=P_{2}(S)$. Further, $\widetilde{f}(c)=\brak{f(\widetilde{c}),f(\tau(\widetilde{c}))}$. Now $f(\widetilde{c})$ (resp.\ $f(\tau(\widetilde{c}))$) is an arc from $f(\widetilde{x_{0}})$ (resp.\ $f(\tau(\widetilde{x_{0}}))$) to an element of $\brak{f(\widetilde{x_{0}}), f(\tau(\widetilde{x_{0}}))}$. But $\widetilde{f}_{\#}(\ang{c})\in P_{2}(S)$, so 
$f(\widetilde{c})$ (resp.\ $f(\tau(\widetilde{c}))$) is a loop based at $f(\widetilde{x_{0}})$ (resp.\ $f(\tau(\widetilde{x_{0}}))$). Thus $\widetilde{c}$ could not be an arc from $\widetilde{x_{0}}$ to $\tau(\widetilde{x_{0}})$, for otherwise $\widetilde{x_{0}}\in X$ would satisfy $f(\widetilde{x_{0}})=f(\tau(\widetilde{x_{0}}))$, which contradicts the hypothesis. Hence $\widetilde{c}$ is a loop based at $\widetilde{x_{0}}$, so $\ang{\widetilde{c}}\in \pi_1(X,\widetilde{x_{0}})$, and $\ang{c}=p_{\#}(\ang{\widetilde{c}})$. Thus $\gamma=\theta_{\tau}(\ang{c})=\theta_{\tau}\circ p_{\#}(\ang{\widetilde{c}})=\overline{0}$, and $\rho$ is injective, as claimed, so is an isomorphism. Taking $\phi=\widetilde{f}_{\#}$ yields the required conclusion.

We now prove the converse for the three cases~(\ref{it:propcase1})--(\ref{it:propcase3}) of the second part of the proposition. Suppose that there exists a homomorphism $\map{\phi}{\pi_1(X/\tau)}{B_2(S)}$ that makes the diagram~(\ref{eq:basic}) commute. We treat the three cases of the statement in turn.
\begin{enumerate}[(a)]
\item By replacing each group $G$ in the algebraic diagram~(\ref{eq:basic}) by the space $K(G, 1)$, we obtain a diagram of spaces that is commutative up to homotopy.  The first possible non-vanishing homotopy group of the fibre of the classifying map $D_2(S) \to K(B_2(S),1)$ of the universal covering  of $D_2(S)$ is in dimension greater than or equal to two. Since $X$ is of dimension at most two,  by classical obstruction theory~\cite[Chapter~V, Section~4, Theorem~4.3, and Chapter~VI, Section~6, Theorem~6.13]{Wh},  
there exists a map $\map{\widetilde{f}}{X/\tau}{D_2(S)}$ that induces $\phi$ on the level of fundamental groups. The composition of a lifting to the double coverings $X \to F_2(S)$ of the map $\widetilde{f}$ with the projection onto the first coordinate of $F_2(S)$ gives rise to a map that does not collapse any orbit to a point, and the result follows. 

\item  Since $S$ is different from $\St[2]$ and $\rp$, the space $D_2(S)$ is a $K(\pi, 1)$, so all of its higher homotopy groups vanish. Arguing as in case~(\ref{it:propcase1}),  there is no obstruction to constructing a map $\widetilde{f}$ that induces $\phi$ on the level of fundamental groups, which proves the result in this case.  

\item Suppose that $S=\rp$. By~\cite{GG2}, it follows that the universal covering of $D_2(\rp)$ has the homotopy type of the $3$-sphere. Since $X$ has dimension less than or equal to three, using classical obstruction theory, we may construct a map $\widetilde{f}$ that satisfies the conditions, and once more the result follows. \qedhere
\end{enumerate}
\end{proof}

Proposition~\ref{fin} is an immediate consequence of the first part of Proposition~\ref{gen} above.

\begin{proof}[Proof of Proposition~\ref{fin}.]
The finiteness of $\pi_1(X)$ implies that of $\pi_1(X/\tau)$. Since $B_2(S)$ is torsion free, there is no factorisation $\phi$ of the algebraic diagram~(\ref{eq:basic}) of Proposition~\ref{gen}, and the result follows.
\end{proof}


\begin{rem}
If $S$ is $\St[2]$ (resp.\ $\rp$), the difficulty in proving the converse in the case $\dim(X)>2$ (resp.\ $\dim(X)>3$) occurs as a result of the non-vanishing of the higher homotopy groups of the $2$-sphere (resp.\ the $3$-sphere). 
\end{rem}

If $S$ is a compact, connected surface without boundary, by Proposition~\ref{gen}(\ref{it:propcase2}), there are two possibilities for $S$ where we do not have equivalence with the existence of a factorisation of the diagram~(\ref{eq:basic}). The case of $\rp$ will be treated in Section~\ref{sec:rp2}. For now, let us consider the case where the target is the sphere $\St[2]$.

\begin{prop}\label{equi} If $X$ is a $CW$-complex equipped with a cellular free involution $\tau$, a triple $(X, \tau, \St[2])$ satisfies the Borsuk-Ulam property if and only if the  classifying  map $\map{g}{X/\tau}{K(\Z_2, 1)}$  of the double covering $X \to X/\tau$ does not factor $($up to homotopy$)$ through the inclusion $\rp \to \R P^{\infty}=K(\Z_2, 1)$. 
\end{prop}

\begin{proof}
First note that the space $D_2(\St[2])$ has the homotopy type of $\rp$~\cite{GG2,GG4}.  
If there is a factorisation of $g$ (up to homotopy) through the inclusion $\rp \to \R P^{\infty}$ then we may construct a map $\map{g_1}{X/\tau}{D_2(\St[2])}$.
Consequently, there exists a $\Z_2$-equivariant lifting $\map{\widetilde{g}_1}{X}{F_2(\St[2])}$.
The composition of $\widetilde{g}_1$ with the projection onto the first coordinate of $F_2(\St[2])$ is a map
for which the Borsuk-Ulam property does not hold. Conversely, if the Borsuk-Ulam property does not
hold for the triple $(X, \tau, \St[2])$ then by a routine argument, the map which does not collapse any orbit gives rise to the factorisation. 
\end{proof}


We are now able to prove Corollary~\ref{2dcw}.

\begin{proof}[Proof of Corollary~\ref{2dcw}.]\mbox{}
\begin{enumerate}[(a)]
\item  By Proposition~2.2(iv) of~\cite{GHZ}, $(X, \tau, \R^3)$ satisfies the Borsuk-Ulam property if and only if  there is no map $\map{f}{X/\tau}{\rp}$ such that the pull-back of 
the non-trivial class of $H^1(\R P^{2}; \Z_2)$ is the first characteristic class of the $\Z_2$-bundle $X \to X/\tau$. 
But this is exactly the condition given by 
Proposition~\ref{equi} for $(X, \tau, \St[2])$.

\item Since the homomorphism $B_2(\St[2]) \to \Z_2$ is an isomorphism, the result follows from Proposition~\ref{gen}.
\item  and \emph{(\ref{it:2dcw3d})}. This is a consequence of the main result of~\cite{GNS}. The fact that there is only one involution on $\mathbb{R} P^3$ up to conjugacy follows from~\cite{My}. \qedhere
\end{enumerate}
\end{proof}

\begin{rem}
The `if' part of Corollary~\ref{2dcw}(\ref{it:2dcw3}) can also be proved by a very simple geometrical argument. For the converse, we do not know of a more direct proof. One may find other examples of triples such as those given in Corollary~\ref{2dcw}(\ref{it:2dcw3c}), i.e.\ triples $(X , \tau, \St[2])$, where $X$ is a $CW$-complex of dimension $3$, for which the Borsuk-Ulam property holds. See~\cite{GNS} for more details.
\end{rem}

To conclude this section, recall from the Introduction that if we are given a group $G$, two surjective homomorphisms $\phi_1,\phi_2\in  \operatorname{Hom}(G,\Z_2)$ are said to be \emph{equivalent} if there is an isomorphism $\map{\varphi}{G}{G}$ such that $\phi_1\circ \varphi=\phi_2$. We shall see that many algebraic questions will depend only on the equivalence classes of this relation due to the fact that if $\phi_1,\phi_2$ are equivalent then the existence of the commutative diagram~(\ref{eq:basic}) for $\phi_{1}$ is equivalent to the existence of the commutative diagram~(\ref{eq:basic}) for $\phi_{2}$. A consequence of this is that the number of cases to be analysed may be reduced. From the Appendix, we have the following results:
\begin{enumerate}[$($a$)$]
\item If $G$ is isomorphic to the fundamental group of an orientable compact, connected surface without boundary and of genus greater than zero then there is precisely one equivalence class.
\item  Suppose that $G$ is isomorphic to the fundamental group of the non-orientable surface $N_{l}$, where $l>1$.  
\begin{enumerate}[(i)]
\item If $l\neq 2$, there are three distinct equivalence classes.
\item If $l=2$, there are two distinct equivalence classes.
\end{enumerate}
\end{enumerate}
The knowledge of these classes will be used in conjunction with Proposition~\ref{gen}, notably in Section~\ref{sec:ornotS2}, to study the validity of the Borsuk-Ulam property.

\section{The case $S=\rp$}\label{sec:rp2}

In this section, we study the second exceptional case of Proposition~\ref{gen}(\ref{it:propcase2}) where the target $S$ is the projective plane $\rp$. Indeed, by the proof of the first part of Proposition~\ref{gen}, a triple $(X, \tau, \rp)$ does not satisfy the Borsuk-Ulam property if and only if there exists a map $\map{\widetilde{f}}{X/\tau}{D_2(\rp)}$ for which the choice $\phi=\widetilde{f}_{\#}$ makes the diagram~(\ref{eq:basic}) commute. Recall that $B_{2}(\rp)$ is isomorphic to the generalised quaternion group $\quat[16]$ of order $16$~\cite{vB}.

\begin{prop}  
Given the notation of Proposition~\ref{gen}, the non-existence of a factorisation $\map{\phi}{\pi_1(X/\tau)}{\quat[16]}$ of the homomorphism $\map{\theta_{\tau}}{\pi_1(X/\tau)}{\Z_2}$ through the 
homomorphism $\quat[16] \to \Z_2$ implies that the Borsuk-Ulam property  holds. Conversely, if a factorisation exists, the Borsuk-Ulam property holds if and only if the map $\map{f_1}{X/\tau}{K(\quat[16], 1)}$ obtained from the algebraic homomorphism $\phi$ does not factor through the map
$\St[3]/\quat[16] \to K(\quat[16], 1)$ given by the Postnikov system, where $K(\quat[16], 1)$ is the first stage of the  Postnikov tower of $\St[3]/\quat[16]$. In particular, 
if $X$ has dimension less than or equal to three, if the algebraic factorisation problem has a solution then the Borsuk-Ulam property does not hold.
\end{prop}

\begin{proof}
The proof follows straightforwardly from Proposition~\ref{gen}.
\end{proof}

Now we can prove the main result of this section.

\begin{proof}[Proof of Theorem~\ref{th:rp2case}.]
Since $X$ is of dimension less than or equal to three, the result is equivalent to the
existence of the homomorphism $\phi$ by Proposition~\ref{gen}(\ref{it:propcase3}). Suppose first that $X$ is simply connected. Then the fundamental group of the quotient $X/\tau$ is isomorphic to $\Z_2$. Since the only element of $B_2(\rp)$ of order $2$ is the full twist braid, which belongs to $P_2(\rp)$, the factorisation of diagram~(\ref{eq:basic}) does not exist, and this proves the `if' part.

Conversely, suppose that $X$ is non-simply connected. Then the fundamental group of $X/\tau$ is either isomorphic to the fundamental group of $S_g$, where $g>0$, or is isomorphic to the fundamental group of $N_l$, where $l>1$ (recall that $S_g$ (resp.\ $N_{l}$) is a compact, connected orientable (resp.\ non-orientable) surface without boundary of genus $g$ (resp.\ $l$)). Let us first prove the result in the case where $\pi_1(X/\tau)\cong\pi_1(S_g)$. The fundamental group of $S_g$ has the following presentation:
\begin{equation}\label{presor}
\setangr{a_1,a_2,\ldots,a_{2g-1},a_{2g}}{[a_1,a_2]\cdots[a_{2g-1},a_{2g}]\,}. 
\end{equation}
Consider the presentation $\setangl{x,y}{x^4=y^2,\; yxy^{-1}=x^{-1}}$ of $\quat[16]$. Then $x$ is of order $8$, and defining 
\begin{equation}\label{eq:defai}
\phi(a_i)=
\begin{cases}
x & \text{if $\theta_{\tau}(a_i)=\overline{1}$}\\
x^2 & \text{if $\theta_{\tau}(a_i)=\overline{0}$}
\end{cases}
\end{equation}
gives rise to a factorisation. Now suppose that $\pi_1(X/\tau)\cong\pi_1(N_{l})$. If $l\geq 3$ is odd, $\pi_1(N_l)$ has the following presentation:
\begin{equation}\label{eq:presnonorodd}
\setangr{v,a_1,a_2,\ldots,a_{l-2},a_{l-1}}{v^2\cdot [a_1,a_2]\cdots[a_{l-2},a_{l-1}]}.
\end{equation}
If $\theta_{\tau}(v)=\overline{0}$ then we define $\phi$ by $\phi(v)=e$ (the trivial element of $B_{2}(\rp)$), and $\phi(a_{i})$ by \req{defai}. If $\theta_{\tau}(v)=\overline{1}$ then we define $\phi(v)=xy$. Now $\phi(v^2)=x^4$ which is of order $2$, and so $\phi(v^2)$ is the full twist braid. Defining
\begin{equation*}
\begin{cases}
\text{$\phi(a_{1})=x^7y$ and $\phi(a_{2})=xy$} & \text{if $\theta_{\tau}(a_{1})= \theta_{\tau}(a_{2})=\overline{1}$}\\
\text{$\phi(a_{1})=x^2$ and $\phi(a_{2})=y$} & \text{if $\theta_{\tau}(a_{1})= \theta_{\tau}(a_{2})=\overline{0}$}\\
\text{$\phi(a_{1})=xy$ and $\phi(a_{2})=x^2$} & \text{if $\theta_{\tau}(a_{1})=\overline{1}$ and $ \theta_{\tau}(a_{2})=\overline{0}$}\\
\text{$\phi(a_{1})=x^2$ and $\phi(a_{2})=xy$} & \text{if $\theta_{\tau}(a_{1})=\overline{0}$ and $ \theta_{\tau}(a_{2})=\overline{1}$,}
\end{cases}
\end{equation*}
and the remaining $\phi(a_{i})$ by \req{defai}, we obtain a factorisation of the commutative diagram~(\ref{eq:basic}), and the result follows. The case where $l\geq 2$ is even is similar, and is left to the reader.
\end{proof}

\section{The non-orientable case with $S\neq \rp$}\label{sec:rootspns}

In this section, we consider the case where the target $S$ is a  compact, connected non-orientable surface without boundary and different from $\rp$. Recall 
that $\pi_1(X/\tau)$ is isomorphic to the fundamental group of a compact, connected surface without boundary. In this section, we prove Theorems~\ref{orien} and~\ref{th:rootspns}, which is the case where this surface is orientable or non orientable respectively.


\begin{proof}[Proof of Theorem~\ref{orien}.]
Let $h\geq 1$ be such that $\pi_1(X/\tau)\cong \pi_{1}(S_{h})$, and consider the presentation~(\ref{presor}) of $\pi_1(X/\tau)$. Let $x \in B_2(S)\setminus P_2(S)$. Then we define
\begin{equation}\label{eq:defaigen}
\phi(a_i)=
\begin{cases}
x & \text{if $\theta_{\tau}(a_i)=\overline{1}$}\\
x^2 & \text{if $\theta_{\tau}(a_i)=\overline{0}$.}
\end{cases}
\end{equation}
The fact that the relation of $\pi_1(X/\tau)$ is given by a product of commutators implies that $\phi$ is a well-defined homomorphism that makes the diagram~(\ref{eq:basic}) commute. The result then follows by applying Proposition~\ref{gen}(\ref{it:propcase2}).
\end{proof}

We now suppose that $\pi_1(X/\tau)$ is isomorphic to the fundamental group of  the non-orientable surface $N_{l}$. 

\begin{proof}[Proof of Theorem~\ref{th:rootspns}.]
The `if' part follows because $\pi_1(X/\tau)\cong \Z_2$ and $B_2(S)$ 
is torsion free. Indeed, there is no algebraic factorisation of the diagram~(\ref{eq:basic}) since the only homomorphism that makes the diagram commute is the trivial homomorphism. For the `only if' part, let $S=N_{m}$, where $m\geq 2$, and let $\pi_1(X/\tau)$ be isomorphic to the fundamental group of the non-orientable surface $N_{l}$, where $l\geq 2$. We first suppose that $l$ is even. Then  $\pi_1(X/\tau)$ has the following presentation:
\begin{equation}\label{eq:presnorleven}
\setangr{\alpha, \beta, a_1,a_2,\ldots,a_{2l-3},a_{2l-2}}{\alpha\beta\alpha\beta^{-1}[a_1,a_2] \cdots
[a_{2l-3},a_{2l-2}]}.
\end{equation}
From~\cite{S}, we have the following relations in the braid group $B_2(N_m)$:  
$\rho_{2,1}\rho_{1,1}\rho_{2,1}^{-1}=\rho_{1,1}B^{-1}$, $B=\sigma^2$,
$\sigma\rho_{1,1}\sigma^{-1}=\rho_{2,1}$ and 
$\sigma\rho_{2,1}\sigma^{-1}=B\rho_{1,1}B^{-1}$ (here $\sigma$ denotes the generator $\sigma_{1}$). We remark that the given elements of $B_2(N_m)$ are those of~\cite{S}, but we choose to multiply them from left to right, which differs from the convention used in~\cite{S}. Other presentations of braid groups of non-orientable surfaces may be found in~\cite{Be,GG3}
Now $\rho_{2,1}\rho_{1,1}\rho_{2,1}^{-1}=\rho_{1,1}B^{-1}$ implies that
 $\rho_{2,1}\rho_{1,1}\rho_{2,1}^{-1}B\rho_{1,1}^{-1}B^{-1}=B^{-1}$. Using the equation $\sigma\rho_{1,1}^{-1}\rho_{2,1}^{-1}\sigma^{-1}=(\sigma\rho_{1,1}^{-1}\sigma^{-1})(\sigma\rho_{2,1}^{-1}\sigma^{-1})=\rho_{2,1}^{-1}B\rho_{1,1}^{-1}B^{-1}$, this implies in turn that 
$\rho_{2,1}\rho_{1,1}\sigma\rho_{1,1}^{-1}\rho_{2,1}^{-1}\sigma^{-1}=B^{-1}=\sigma^{-2}$, and hence $\rho_{2,1}\rho_{1,1}\sigma\rho_{1,1}^{-1}\rho_{2,1}^{-1}= \sigma^{-1}$.

Now we construct the factorisation $\phi$. If $\theta_{\tau}(\alpha)=\overline{0}$ then define
$\phi(\alpha)=e$,  and $\phi(\beta)$ to be equal to any element of $B_2(S)\setminus P_2(S)$ 
if $\theta_{\tau}(\beta)= \overline{1}$, and to be equal to $e$ if $\theta_{\tau}(\beta)= \overline{0}$.  If $\theta_{\tau}(\alpha)= \overline{1}$ and $\theta_{\tau}(\beta)= \overline{0}$  then we define 
$\phi(\alpha)=\sigma$ and $\phi(\beta)=\rho_{2,1}\rho_{1,1}$, while if 
$\theta_{\tau}(\alpha)= \theta_{\tau}(\beta)=\overline{1}$, we define 
$\phi(\alpha)=\sigma$ and $\phi(\beta)=\rho_{2,1}\rho_{1,1}\sigma$.
For the remaining generators $a_i$, we define $\phi$ as in \req{defaigen}. It follows from the construction that $\phi$ is a well-defined homomorphism that makes the diagram~(\ref{eq:basic}) commute.

Finally let the fundamental group $\pi_1(X/\tau)$ be isomorphic to $\pi_{1}(N_{l})$, where $l\geq 3$ is odd. Consider the presentation~(\ref{eq:presnonorodd}) of $\pi_{1}(N_{l})$.
If $\theta_{\tau}(v)=\overline{0}$ then the result follows as in the proof of Theorem~\ref{orien}.
So suppose that $\theta_{\tau}(v)=\overline{1}$.  We have the relation $\rho_{2,1}B\rho_{2,1}^{-1}=B\rho_{1,1}^{-1}B^{-1}\rho_{1,1}B^{-1}$

According to Proposition~\ref{prop:appendor} in the Appendix it suffices to consider two cases. The first 
is $\theta_{\tau}(a_i)=\overline{0}$ for all $i$; the second is $\theta_{\tau}(a_2)=\overline{1}$ and 
$\theta_{\tau}(a_i)=\overline{0}$ for the other values of $i$. In the first case, we define $\phi(v)=\sigma$,
$\phi(a_1)= \rho_{1,1}^{-1}$, $\phi(a_2)= \rho_{2,1}$      and
for the remaining generators $a_i$, we define $\phi(a_{i})$ as in \req{defaigen}. The result follows via the relation of the presentation~(\ref{eq:presnonorodd}). As for the second case, we define $\phi(v)=\sigma$,  $\phi(a_1)=\sigma^{-1} $, $\phi(a_2)=\rho_{2,1}\rho_{1,1}$. 
and for the remaining $a_i$, we define 
$\phi(a_{i})$ as in \req{defaigen}. The result then follows.
\end{proof}

\section{The orientable case with $S\ne \St[2]$}\label{sec:ornotS2}

The purpose of this section is to study the Borsuk-Ulam property in the case where the target is a compact, connected orientable surface without boundary of genus greater than zero. This is the most delicate case which we will separate into several subcases.  As in the previous section, $\pi_1(X/\tau)$ is isomorphic to the fundamental group of a compact, connected surface without boundary. We first suppose that this surface is orientable.

\begin{proof}[Proof of Theorem~\ref{orien-orien}]
Similar to that of Theorem~\ref{orien}.
\end{proof}

We now suppose that $\pi_1(X/\tau)$ is isomorphic to the fundamental group 
of the non-orientable surface $N_{l}$. Let $S=S_g$, where $g\geq 1$. As we mentioned in the Introduction, we consider the following four subcases.
\begin{enumerate}[(1)]
\item $l=1$.
\item $l \geq 4$. 
\item $l=2$.
\item\label{it:case5ornonor} $l=3$.
\end{enumerate}

As we shall see, the first two cases may be solved easily. The third case is a little more  difficult. The fourth case is by far the most difficult, and will occupy most of this section. Some of the tools used in this last case will appear in the discussion of the first three cases. Let us now study these cases in turn.

\begin{enumerate}[Subc{a}se (1):]
\item $l=1$. This is the subcase where $\pi_1(X/\tau)$ is isomorphic to the fundamental group of the 
projective plane $\rp$.

\begin{proof}[Proof of Proposition~\ref{case1}.]
Since $B_2(S_g)$ is non trivial and torsion free, it follows that there is no algebraic factorisation of the diagram~(\ref{eq:basic}), and the result follows from Proposition~\ref{gen}. 
\end{proof}

\item $l \geq 4$. We recall a presentation of $P_{2}(S_{g})$ that may be found in~\cite{FH} and that shall be used at various points during the rest of the paper. Other presentations of $P_{2}(S_{g})$ may be found in~\cite{Be,GG1}.

\begin{thm}[\cite{FH}]\label{th:fadhu}
Let $g\geq 1$. The following is a presentation of $P_{2}(S_{g})$.
\begin{enumerate}
\item[\textbf{generators:}] $\rho_{i,j}$, where $i=1,2$ and $j=1,\ldots,2g$.
\item[\textbf{relations:}]\mbox{}
\begin{enumerate}[(I)]
\item\label{it:fadhuI} $[\rho_{1,1},\rho_{1,2}^{-1}]\cdots[\rho_{1,2g-1},
\rho_{1,2g}^{-1}]=B_{1,2}=B_{2,1}^{-1}=[\rho_{2,1},\rho_{2,2}^{-1}] \cdots
[\rho_{2,2g-1},\rho_{2,2g}^{-1}]$ (this defines the elements $B_{1,2}$ and $B_{2,1}^{-1}$).

\item $\rho_{2,l}\rho_{1,j}=\rho_{1,j}\rho_{2,l}$ where $1\leq j,l\leq 2g$, and $j<l$ (resp.\ $j<l-1$) if $l$ is odd (resp.\ $l$ is even).

\item $\rho_{2,k}\rho_{1,k}\rho_{2,k}^{-1}=\rho_{1,k}[\rho_{1,k}^{-1},B_{1,2}]$
and 
$\rho_{2,k}^{-1}\rho_{1,k}\rho_{2,k}=\rho_{1,k}[B_{1,2}^{-1},\rho_{1,k}]$ for all $1\leq k\leq 2g$.

\item $\rho_{2,k}\rho_{1,k+1}\rho_{2,k}^{-1} =B_{1,2}\rho_{1,k+1}
[\rho_{1,k}^{-1},B_{1,2}]$, and $\rho_{2,k}^{-1}\rho_{1,k+1}\rho_{2,k} =B_{1,2}^{-1}[B_{1,2},\rho_{1,k}]\rho_{1,k+1}
[B_{1,2}^{-1}, \rho_{1,k}]$, for all $k$ odd, $1\leq k\leq 2g$.

\item\label{it:fadhuV} $\rho_{2,k+1}\rho_{1,k}\rho_{2,k+1}^{-1}=\rho_{1,k}B_{1,2}^{-1}$, and
 $\rho_{2,k+1}^{-1}\rho_{1,k}\rho_{2,k+1}= \rho_{1,k}B_{1,2}[B_{1,2}^{-1},
\rho_{1,k+1}]$,  for all $k$ odd, $1\leq k\leq 2g$.

\item $\rho_{2,l}\rho_{1,j}\rho_{2,l}^{-1}=[B_{1,2},\rho_{1,l}^{-1}]\rho_{1,j}
[\rho_{1,l}^{-1}, B_{1,2}]$ and 
 $\rho_{2,l}^{-1}\rho_{1,j}\rho_{2,l}=[ \rho_{1,l},B_{1,2}^{-1}]\rho_{1,j}
[B_{1,2}^{-1},\rho_{1,l}]$ for all $1\leq l<j\leq 2g$ and $(j,l) \neq (2t,2t-1)$ for all $t\in \brak{1,\ldots,g}$.
\end{enumerate}
\end{enumerate}
\end{thm}

From the above relations, we obtain 
\begin{equation}\label{eq:fha}
\rho_{2,k}B_{1,2}\rho_{2,k}^{-1}= B_{1,2}\rho_{1,k}^{-1}B_{1,2}\rho_{1,k}B_{1,2}^{-1},
\end{equation}
and $\rho_{2,k}^{-1}B_{1,2}\rho_{2,k}=\rho_{1,k}B_{1,2}\rho_{1,k}^{-1}$. Let $\sigma=\sigma_{1}$ be the standard generator of $B_{2}(S_{g})$ that swaps the two basepoints, and set $B=B_{1,2}=\sigma^2$. The crucial relation that we shall require is 
\begin{equation*}
\text{ $\rho_{2,2i}\rho_{1,2i-1}\rho_{2,2i}^{-1}=\rho_{1,2i-1}B^{-1}$,  where $i\in \brak{1,\ldots,g}$.}
\end{equation*}



\begin{proof}[Proof of Proposition~\ref{case2}.] 
First assume that $l$ is odd. Then $N_{l}$ has the presentation given by \req{presnonorodd}. Using Proposition~\ref{prop:appendor}, for at least two 
generators $a_{2i-1}, a_{2i}$ with $1<i\leq g$, we have $\theta_{\tau}(a_{2i-1})= \theta_{\tau}(a_{2i})=\overline{0}$. If $\theta_{\tau}(v)=\overline{0}$ then 
the factorisation is defined as in the corresponding case of the proof of 
Theorem~\ref{orien}. So assume that  $\theta_{\tau}(v)=\overline{1}$, and define  
 $\phi(v)=\sigma$,  
$\phi(a_{2i-1})=\rho_{1,1}^{-1}$,  $\phi(a_{2i})=\rho_{2,2}$,
and for $j\notin \brak{2i-1,2i}$, set $\phi(a_j)=\sigma$ if $\theta_{\tau}(a_{j})=\overline{1}$, and $\phi(a_j)=e$ if $\theta_{\tau}(a_{j})=\overline{0}$. It follows from the relation of the presentation of $\pi_{1}(N_{l})$ given in equation~(\ref{eq:presnonorodd}) and the first relation of~(V) of Theorem~\ref{th:fadhu}
that $\phi$ is a well-defined homomorphism that makes the diagram~(\ref{eq:basic}) commute. 
The result follows from Proposition~\ref{gen}.

If $l\geq 4$ is even, the proof is similar. Once more, from Proposition~\ref{prop:appendor}, we have $\theta_{\tau}(a_{2i-1})= \theta_{\tau}(a_{2i})=\overline{0}$ for some $i\in \brak{1,\ldots, g}$. 
The fundamental group of the surface $N_{l}$ has the presentation given by \req{presnorleven}. Define $\phi(\alpha)=\sigma$, and set $\phi(\beta)=e$ if $\theta_{\tau}(\beta)=\overline{0}$ and $\phi(\beta)=\sigma$ if $\theta_{\tau}(\beta)=\overline{1}$. We define $\theta_{\tau}(a_{i})$ as in the case $l$ odd, and the result follows in a similar manner.
\end{proof}

Before going any further, we define some notation that shall be used to discuss the remaining two cases. For $i=1,2$, the two projections $\map{p_i}{P_2(S_g)}{P_1(S_g)}$ furnish a homomorphism $\map{p_1\times p_2}{P_2(S_g)}{P_1(S_g)\times P_1(S_g)}$ (which is the homomorphism induced by the inclusion $F_2(S_g) \to S_g\times S_g$). Let $N$ denote the kernel of $p_1\times p_2$. We thus have a short exact sequence 
\begin{equation}\label{eq:sesp2g}
1 \to N \to P_{2}(S_{g}) \stackrel{p_1\times p_2}{\to} P_1(S_g)\times P_1(S_g) \to 1.
\end{equation}
Let $(x_{1},x_{2})$ be a basepoint in $F_{2}(S_{g})$, let
\begin{equation}\label{eq:deffi}
\text{$\mathbb{F}_{1}=P_{1}(S_{g}\setminus \brak{x_{2}}, x_{1})$, and let $\mathbb{F}_{2}=P_{1}(S_{g}\setminus \brak{x_{1}}, x_{2})$.}
\end{equation}
We know that for $i=1,2$, $\mathbb{F}_{i}=\ker{p_{j}}$, where $j\in \brak{1,2}$ and $j\neq i$, and that $\mathbb{F}_{i}$ is a free subgroup of $P_{2}(S_{g})$ of rank $2g$ with basis $\brak{\rho_{i,1},\ldots,\rho_{i,2g}}$.
Now $N$ is also equal to the normal closure of $B$ in $P_{2}(S_{g})$ (see~\cite{FH}, and Proposition~3.2 in particular), and is a free group of infinite rank with basis $\setr{B_{\eta}=\eta B
\eta^{-1}}{\eta \in \mathbb{S}_{1}}$, where $\mathbb{S}_{1}$ is a
Reidemeister-Schreier system for the projection
$\pi_1(S_g \setminus \brak{x_{2}},x_{1})\to\pi_1(S_g,x_{1})$. 

\item $l=2$. Suppose that $\pi_1(X/\tau)\cong\pi_1(K)$, where $K$ denotes the Klein bottle.

\begin{proof}[Proof of Proposition~\ref{case4}.]
If $\theta_{\tau}(\alpha)=\overline{0}$, it is straightforward to check that we have a factorisation of diagram~(\ref{eq:basic}), and so by Proposition~\ref{gen}, the Borsuk-Ulam 
property does not hold for the triple $(X,\tau,S_{g})$. Conversely, assume that $\theta_{\tau}(\alpha)=\overline{1}$, and suppose that the Borsuk-Ulam property does not hold for the triple $(X,\tau,S_{g})$. We will argue for a contradiction. Since $\theta_{\tau}(\alpha)=\overline{1}$ we may assume by Proposition~\ref{prop:appendor} that $\theta_{\tau}(\beta)=\overline{0}$. By Proposition~\ref{gen}, we have a factorisation as in diagram~(\ref{eq:basic}). So there are elements which by abuse of notation we also denote $\alpha, \beta\in B_2(S_g)$ satisfying $\beta\alpha\beta^{-1}=\alpha^{-1}$. This relation implies that
\begin{equation}\label{eq:relnp2g}
\beta\alpha^2\beta^{-1}=\alpha^{-2},
\end{equation}
of which both sides belong to $P_2(S_g)$. Applying this homomorphism to \req{relnp2g}, we obtain two similar equations, each in $P_1(S_g)$. For each of these two equations, the subgroup of $P_{1}(S_{g})$ generated by $p_{i}(\alpha^2)$ and $p_{i}(\beta)$, for $i=1,2$, 
must necessarily have rank at most one (the subgroup is free Abelian if $g=1$, and is free if $g>1$, so must have rank one as a result of the relation). This implies that $p_{i}(\alpha^2)$ is trivial. Therefore $\alpha^2 \in N$. 
The Abelianisation $N_{\text{Ab}}$ of
$N$ is isomorphic to the group ring $\Z [\pi_1(S_g)]$, by means of the
natural bijection $\mathbb{S}_{1} \to \pi_1(S_g)$. Let
$\map{\lambda}{N}{N_{\text{Ab}}}$ denote the
Abelianisation homomorphism, and let $\map{\operatorname{exp}}
{\Z[\mathbb{S}_{1}]}{\Z}$ denote the evaluation homomorphism.

Since $\alpha^2\in N$, both sides of \req{relnp2g} belong to $N$. Equation~(\ref{eq:fha}) implies that $\operatorname{exp} \circ \lambda(\beta \alpha^2 \beta^{-1})=\operatorname{exp} \circ \lambda(\alpha^{2})$, and so $\operatorname{exp} \circ \lambda(\alpha^{2})=0$ by \req{relnp2g}. On the other hand, $\alpha\in B_{2}(S_{g})\setminus P_{2}(S_{g})$, so there exists $\gamma\in P_{2}(S_{g})$ satisfying $\alpha=\gamma \sigma$. Hence 
\begin{equation}\label{eq:alpha2odd}
\alpha^2 =\gamma \sigma \gamma \sigma^{-1} \ldotp B,
\end{equation}
and since $\alpha^2, B\in N$, we see that $\gamma \sigma \gamma \sigma^{-1}\in N$. Now $\gamma\in P_{2}(S_{g})$, so we may write $\gamma=w_{1}w_{2}$, where for $i=1,2$, $w_{i}\in \mathbb{F}_{i}$. Setting $w_{i}'=\sigma w_{i} \sigma^{-1}$ for $i=1,2$, we have that $w_{i}'\in \mathbb{F}_{j}$, where $j$ satisfies $\brak{i,j}=\brak{1,2}$. 
Further, $1=(p_{1}\times p_{2})(w)=(p_{1}\times p_{2})(w_{1}w_{2}
w_{1}'w_{2}')=(w_{1}w_{2}', w_{2}w_{1}')$ (we abuse notation slightly by writing the elements of the factors of $P_{1}(S_{g})\times P_{1}(S_{g})$ in the same form as the corresponding elements of $P_{2}(S_{g})$). 
Thus $w_{1}w_{2}'$ and $ w_{1}'w_{2}$, considered as elements of $P_{2}(S_{g})$, belong to $N$. We have that
\begin{equation*}
\sigma w_{1}w_{2}' \sigma^{-1}= w_{1}' \ldotp \sigma w_{2}' \sigma^{-1}= w_{1}' B w_{2} B^{-1}= w_{1}' w_{2}\ldotp w_{2}^{-1}B w_{2} \ldotp B^{-1},
\end{equation*}
and since $\exp\circ \lambda(\sigma w_{1}w_{2}' \sigma^{-1})=\exp\circ \lambda(w_{1}w_{2}')$, it follows that
\begin{equation}\label{eq:explam}
\exp\circ \lambda(w_{1}w_{2}')=\exp\circ \lambda(w_{1}'w_{2}).
\end{equation}
Now
\begin{equation*}
\gamma \sigma \gamma \sigma^{-1}= w_{1}w_{2} w_{1}'w_{2}'= w_{1}w_{2}' \ldotp w_{2}'^{-1}w_{2} (w_{1}'w_{2}) w_{2}^{-1} w_{2}',
\end{equation*}
and thus $\exp\circ \lambda(\gamma \sigma \gamma \sigma^{-1})=2\exp\circ \lambda(w_{1}w_{2}')$ by \req{explam}. In particular, $\exp\circ \lambda(\alpha^2)$ is odd by \req{alpha2odd}, which contradicts the fact that $\operatorname{exp} \circ \lambda(\alpha^{2})=0$. We thus conclude that the equation $\beta\alpha\beta^{-1}=\alpha^{-1}$, where $\alpha\in B_2(S_g)\setminus P_{2}(S_{g})$, $\beta\in P_{2}(S_{g})$, has no solution, and hence the Borsuk-Ulam property holds for the triple $(X,\tau,S_{g})$.
\end{proof}

\item\label{it:subcase4} $l=3$.
Using the results of Proposition~\ref{prop:appendor}, it suffices to consider the following three cases:
\begin{enumerate}[(a)]
\item\label{it:caseaaa} $\theta_{\tau}(v)=\theta_{\tau}(a_2)=\overline{0}$ and $\theta_{\tau}
(a_1)= \overline{1}$. 
\item\label{it:casebb} $\theta_{\tau}(v)= \overline{1}$ and $\theta_{\tau}(a_1)=\theta_{\tau}
(a_2)= \overline{0}$.
\item\label{it:casecc} $\theta_{\tau}(v)=\theta_{\tau}(a_1)=\overline{1}$ and $\theta_{\tau}(a_2)= \overline{0}$. 

\end{enumerate}
\end{enumerate}

Most of the rest of this section is devoted to analysing case~(\ref{it:casecc}), which is by far the most difficult of the three cases. Using the transformations of \repr{appendop}, we may show that case~(\ref{it:casecc}) is equivalent to $\theta_{\tau}(v)=\theta_{\tau}(a_2)=\overline{1}$ and $\theta_{\tau}(a_1)=\overline{0}$, and so by the discussion at the end of \resec{GeBo}, it suffices to consider the latter case. So in what follows, let $\map{\theta_{\tau}}{\pi_1(N_3)}{\Z_{2}}$ be the homomorphism given by  $\theta_{\tau}(v)=\theta_{\tau}(a_2)=\overline{1}$ and $\theta_{\tau}(a_1)=\overline{0}$. We first define some notation. By Proposition~\ref{gen}, we must decide whether there exist $a,c\in B_2(S_g)$ and $w\in P_2(S_g)$ such that
\begin{equation}\label{eq:basic6}
a^2[w,c]=1.
\end{equation}
Set
\begin{equation}\label{eq:subs}
\text{$a=\rho^{-1} \sigma$, and $c=\sigma v$, where $\rho,v\in P_2(S_g)$.}
\end{equation}
In order to determine the existence of solutions to~\req{basic6}, we begin by studying its projection onto $P_1(S_g)\times P_1(S_g)$ via the short exact
sequence~\reqref{sesp2g}, and its projection onto $(P_1(S_g))_{\text{Ab}}\times (P_1(S_g))_{\text{Ab}}$ under the homomorphism
\begin{equation}\label{eq:projab}
P_1(S_g)\times P_1(S_g) \to   (P_1(S_g))_{\text{Ab}}\times
(P_1(S_g))_{\text{Ab}},
\end{equation}
where $(P_1(S_g))_{\text{Ab}}\cong \Z^{2g}$ is the Abelianisation of
$P_1(S_g)$. Since $\rho,v$ and $w$ belong to $P_2(S_g)$, we may write
\begin{equation}\label{eq:subs2}
\text{$\rho=\rho_1\rho_2$, $v=v_1v_2$ and $w=w_1w_2$,}
\end{equation}
where for $i=1,2$, $\rho_i, v_i, w_i \in \mathbb{F}_i$, and
$\mathbb{F}_{i}$ is as defined in \req{deffi}. Given a word $w$ in
$\mathbb{F}_i$ written in terms of the basis $\set{\rho_{i,k}}{1\leq k\leq 2g}$, let $\widetilde{w}$ denote the word in $\mathbb{F}_j$, obtained by replacing each $\rho_{i,k}$ by
$\rho_{j,k}$, where $j\in \brak{1,2}$ and $j\neq i$.  The automorphism $\iota_{\sigma}$ of $P_2(S_g)$ given by conjugation by $\sigma$ has the property that its restriction to $\mathbb{F}_1$ (resp.\ to $\mathbb{F}_2$) coincides with the map that sends $w$ to $\widetilde w$ (resp.\ to $B \widetilde w B^{-1}$). 
The restriction of $\iota_{\sigma}$ to the intersection $\mathbb{F}_1\cap \mathbb{F}_2$, which is the normal closure of $B$, is invariant under $\iota_{\sigma}$. We have
\begin{equation}\label{eq:isigma}
\iota_{\sigma}(w)=\sigma w\sigma^{-1}=
\begin{cases}
\widetilde{w} & \text{if $w \in \mathbb{F}_1$}\\
B\widetilde  w B^{-1} & \text{if $w \in \mathbb{F}_2$,}
\end{cases}
\end{equation}
where in the first (resp.\ second) case, $w$ is written in terms of the basis $\set{\rho_{1,k}}{1\leq k\leq 2g}$ (resp.\ $\set{\rho_{2,k}}{1\leq k\leq 2g}$) of $\mathbb{F}_{1}$ (resp.\ $\mathbb{F}_{2}$).
We will later consider the automorphism induced by $\iota_{\sigma}$ on
a quotient of $P_2(S_g)$ by a term of the lower central series. 

\begin{lem}\label{reduc1}
With the notation introduced above, $\theta_{\tau}$ factors as in
diagram~(\ref{eq:basic}) if and only if for $i=1,2$, there exist $\rho_i,
 v_i, w_i \in \mathbb{F}_i$ such that
\begin{equation}\label{eq:main1}
B= \sigma \rho_1\rho_2\sigma^{-1}\rho_1\rho_2\sigma
v_1v_2w_1w_2v_2^{-1}v_1^{-1}\sigma^{-1}w_2^{-1}w_1^{-1},
\end{equation}
or equivalently, such that
\begin{equation}\label{eq:main2}
B= \widetilde \rho_1 B\widetilde \rho_2 B^{-1}\rho_1\rho_2\widetilde v_1
B\widetilde v_2B^{-1}\widetilde w_1B\widetilde w_2\widetilde
v_2^{-1}B^{-1}\widetilde v_1^{-1}w_2^{-1}w_1^{-1}.
\end{equation}
Furthermore, if we project \req{main2} onto each of the factors of
$P_1(S_g)\times P_1(S_g)$ then the following equations hold in
$P_1(S_g)$:
\begin{equation}\label{eq:w1w2}
\widetilde \rho_2\rho_1\widetilde  v_2 \widetilde w_2\widetilde
v_2^{-1}=w_1\quad \text{and} \quad \widetilde \rho_1\rho_2  \widetilde v_1
\widetilde w_1\widetilde v_1^{-1}=w_2,
\end{equation}
where by abuse of notation, we use the same notation for elements of
$P_2(S_g)$ and their projection in $P_1(S_g)$.
\end{lem}

\begin{proof}
Substituting \req{subs} into \req{basic6} leads to $(\rho^{-1}\sigma)^ 2[w,\sigma v]=1$, which is
equivalent in turn to
$(\rho^{-1}\sigma)(\rho^{-1}\sigma^{-1})\sigma^2[w,\sigma v]=1$, and to
$\sigma^2=\sigma\rho\sigma^{-1}\rho[\sigma v, w]$. Substituting \req{subs2} into this last equation yields \req{main1}. Using \req{isigma}, we obtain \req{main2}. The second part is also straightforward, using the fact
that $\ker{p_{1}}=\mathbb{F}_2$ and $\ker{p_{2}}=\mathbb{F}_1$.
\end{proof}

From \req{sesp2g}, the two equations of \reqref{w1w2} in $P_1(S_g)$ are equivalent respectively to the equations 
\begin{equation}\label{eq:z1z2N}
\text{$\widetilde \rho_2\rho_1\widetilde  v_2 \widetilde w_2\widetilde v_2^{-1}z_1=w_1$ and
$\widetilde \rho_1\rho_2  \widetilde v_1 \widetilde w_1\widetilde v_1^{-1}z_2=w_2$
in $P_2(S_g)$,}
\end{equation}
where $z_1, z_2 \in N$. An easy calculation proves the following:

\begin{lem}\label{lem:prodcomm}
Equation~\reqref{main2} may be rewritten in the form:
\begin{equation}\label{eq:comm}
\begin{aligned}
B =& \left[\widetilde \rho_1, B \widetilde  \rho_2 B^{-1} \rho_1\right]
\left(B\widetilde \rho_2 B^{-1}\rho_1 \left[\widetilde
\rho_1\rho_2\widetilde v_1, B\widetilde v_2
B^{-1}\right]\rho_1^{-1}B\widetilde\rho_2^{-1}B^{-1} \right)\\
& \left(B \widetilde \rho_2B^{-1}\rho_1B\widetilde v_2B^{-1}\left[\widetilde
\rho_1\rho_2\widetilde v_1\widetilde w_1, B\widetilde w_2\widetilde
v_2^{-1}B^{-1}\right]B\widetilde v_2^{-1}B^{-1}\rho_1^{-1}B\widetilde
\rho_2^{-1}B^{-1}\right)\\
& [B\widetilde\rho_2 B^{-1}\rho_1 B\widetilde v_2 \widetilde w_2\widetilde
v_2^{-1}B^{-1},\widetilde \rho_1\rho_2\widetilde v_1\widetilde w_1\widetilde
v_1^{-1} w_2^{-1}]\\
& \left(\widetilde \rho_1\rho_2\widetilde v_1 \widetilde w_1\widetilde
v_1^{-1} w_2^{-1}\right)\left(B\widetilde \rho_2 B^{-1}\rho_1B\widetilde
v_2 \widetilde w_2\widetilde v_2^{-1} B^{-1}w_1^{-1}\right).
\end{aligned}
\end{equation}\par\vspace{-0.7cm}\qed
\end{lem}

\begin{rem}\label{rem:diffn}
Observe that the commutators in \req{comm} have the property
that one of the terms belongs to $\mathbb{F}_1$, while the other
belongs to $\mathbb{F}_2$. Consequently, each commutator belongs to
$N$ by \req{sesp2g}.
\end{rem}

\begin{cor}\label{cor:furthproj}
The elements $\widetilde \rho_1\rho_2\widetilde v_1 \widetilde w_1\widetilde
v_1^{-1} w_2^{-1} $ and $B\widetilde \rho_2 B^{-1}\rho_1B\widetilde v_2
\widetilde w_2\widetilde v_2^{-1} B^{-1}w_1^{-1}$ of $P_{2}(S_{g})$ belong to $N$. If we further project onto the Abelianisation
(cf.~\req{projab}), then the
projections of $\widetilde \rho_2\rho_1$ and  $\widetilde
\rho_1\rho_2$ belong to the commutator subgroup of the factors $P_1(S_g)\times \brak{1}$ and $\brak{1}\times P_1(S_g)$ of $P_1(S_g)\times P_1(S_g)$ respectively.
\end{cor}

\begin{proof}
From \rerem{diffn}, the commutators on
the right-hand side of equation~(\ref{eq:comm})
belong to $N$, and hence the last line of this equation also belongs to $N$. Projecting each of the factors of this last line onto $P_1(S_g)\times P_1(S_g)$ and using \req{z1z2N} yield the first part of the corollary. For the second part, the projection of $\widetilde \rho_1\rho_2\widetilde v_1 \widetilde w_1\widetilde v_1^{-1} w_2^{-1}$ onto the second factor of $(P_1(S_g))_{\text{Ab}}\times (P_1(S_g))_{\text{Ab}}$ via $P_1(S_g)\times P_1(S_g)$ yields $\widetilde \rho_1\rho_2 \widetilde w_1 w_2^{-1}=1$, where once more we do not distinguish notationally between an element of $P_{2}(S_{g})$ and its projection in $(P_1(S_g))_{\text{Ab}}\times (P_1(S_g))_{\text{Ab}}$. Consider $\xi=\widetilde \rho_2\rho_1\widetilde  v_2 \widetilde w_2\widetilde v_2^{-1}w_1^{-1}\in P_{2}(S_{g})$. By \req{z1z2N}, $\xi \in N$. Now $\xi\in \mathbb{F}_{1}$, so $\iota_{\sigma}(\xi)= \rho_2 \widetilde \rho_1  v_2 w_2 v_2^{-1} \widetilde w_1^{-1}$ by \req{isigma}, and since $N$ is equal to the normal closure of $B$ in $P_{2}(S_{g})$, it is invariant under $\iota_{\sigma}$. The projection of $\iota_{\sigma}(\xi)$ onto the second factor of $(P_1(S_g))_{\text{Ab}}\times (P_1(S_g))_{\text{Ab}}$ via $P_1(S_g)\times P_1(S_g)$ thus yields $\rho_2 \widetilde \rho_1 w_2 \widetilde w_1^{-1}=1$. So in this factor of $(P_1(S_g))_{\text{Ab}}$, we have $\widetilde \rho_1\rho_2 \widetilde w_1 =w_2$ and $\rho_2 \widetilde \rho_1 w_2 =\widetilde w_1$. Substituting the second of these equations into the first gives $1=\rho_2 \widetilde \rho_1 \widetilde \rho_1\rho_2=(\widetilde \rho_1\rho_2)^2$ since $(P_1(S_g))_{\text{Ab}}$ is Abelian. The fact that the group
$(P_1(S_g))_{\text{Ab}}\cong \Z^{2g}$ is torsion free implies that $\widetilde \rho_1\rho_2=1$ in $(P_1(S_g))_{\text{Ab}}$. Hence $(1,\widetilde \rho_1\rho_2)$, considered as an element of $\brak{1}\times
P_1(S_g)$ belongs to its commutator subgroup. A similar argument proves the result for $\widetilde \rho_2\rho_1$.
%
\end {proof}

Let $k\in\brak{1,2}$ and $i\in \brak{1,\ldots, 2g}$. Using Theorem~\ref{th:fadhu}, it is not hard to see that if 
$x$ is an element of $P_{2}(S_{g})$ written as a word $w$ in the generators of that theorem then the sum of the exponents of $\rho_{k,i}$ appearing in $w$, which we denote by $\left\lvert x\right\rvert_{\rho_{k,i}}$, is a well-defined integer that does not depend on the choice of $w$. 


\begin{lem}\label{lem:exp} 
Let $i\in \brak{1,\ldots, 2g}$.
\begin{enumerate}[(a)]
\item\label{it:parta} Let $k\in\brak{1,2}$. The map $P_{2}(S_{g}) \to \Z$ given by $x \longmapsto \left\lvert x\right\rvert_{\rho_{k,i}}$ is a homomorphism whose kernel contains $N$.
\item Given a solution of equation (\ref{eq:comm}), we have:
\begin{align*}
\left\lvert \rho_1\right\rvert_{\rho_{1,i}}+\left\lvert \rho_2\right\rvert_{\rho_{2,i}}&=
-\left\lvert w_1\right\rvert_{\rho_{1,i}}+\left\lvert w_2\right\rvert_{\rho_{2,i}}\quad \text{and}\\
\left\lvert \rho_1\right\rvert_{\rho_{1,i}}+\left\lvert \rho_2\right\rvert
_{\rho_{2,i}}&=
\left\lvert w_1\right\rvert_{\rho_{1,i}}-\left\lvert w_2\right\rvert_{\rho_{2,i}}.
\end{align*}
Hence 
\begin{equation}\label{eq:expsumzero}
\text{$\left\lvert \rho_1\right\rvert_{\rho_{1,i}}=-\left\lvert \rho_2\right\rvert_{\rho_{2,i}}$ and $\left\lvert w_1\right\rvert_{\rho_{1,i}}=\left\lvert w_2\right\rvert_{\rho_{2,i}}$.}
\end{equation}
\end{enumerate}
\end{lem}

\begin{proof}\mbox{}
\begin{enumerate}[(a)]
\item follows easily using the presentation of $P_{2}(S_{g})$ given in Theorem~\ref{th:fadhu}.

\item  This is a consequence of applying part~(\ref{it:parta}) to \req{z1z2N}, and using the fact that $\left\lvert x\right\rvert_{\rho_{1,i}}= \left\lvert \widetilde x\right\rvert_{\rho_{2,i}}$ for all $x\in P_{2}(S_{g})$.\qedhere
\end{enumerate}
\end{proof}


Let  $G=P_2(S_g)$, and for $i\in\N$, let $\Gamma_i(G)$ denote the terms of its lower central series. Recall that by definition,  $\Gamma_1(G)=G$ and
$\Gamma_{i+1}(G)=[\Gamma_i(G),G]$ for all $i\in \N$. By \reco{furthproj}, \req{comm} may be interpreted as a relation in $\Gamma_2(P_2(S_g))$. We shall study this equation by means of its projection onto $K \otimes \mathbb{Z}_2$, where $K$ is a certain quotient of $\Gamma_2(G)/\Gamma_3(G)$, which we shall define presently. We first recall some properties of $G/\Gamma_3(G)$.

\begin{lem}\label{lem:commu1}
We have the following relations in the group $G/\Gamma_3(G)$:
\begin{enumerate}[(a)]
\item\label{it:gamma3a} $[ab,c]$=$[a,c]\, [b,c]$, $[a,bc]$=$[a, b]\, [a,c]$ and $[a^s,b^t]=[a,b]^{st}$ for all $a,b,c\in G/\Gamma_3(G)$ and all $s,t\in \Z$.
\item\label{it:gamma3b} The automorphism of $G/\Gamma_3(G)$ induced by $\iota_{\sigma}$ is given by the map which sends the class of a word $w$ in $G$ to the class of the word $\widetilde{w}$.
\item\label{it:gamma3c} Let $1\leq i,j\leq 2g$. In $G/\Gamma_3(G)$ we have that $[\rho_{2,i+1},\rho_{1,i}]=B^{-1}$ and $[\rho_{2,i},\rho_{1,i+1}]=B$ for $i$ odd, and $[\rho_{2,i},\rho_{1,j}]=1$ otherwise (notationally, we do not distinguish between an element of $G$ and its class in $G/\Gamma_3(G)$).
\end{enumerate}
\end{lem}

\begin{proof} Part~(\ref{it:gamma3a}) is a consequence of the well-known
formulas $[ab,c]=a[b,c]a^{-1}[a,c]$ and  $[a, bc]=[a, b] b[a,c] b^{-1}$, 
and the fact that $\Gamma_2(G)/\Gamma_3(G)$ is central in $G/\Gamma_3(G)$. The fact that $[a^s,b^t]=[a,b]^{st}$ then follows by an inductive argument. Part~(\ref{it:gamma3b}) is a consequence of the description of the automorphism $\iota_{\sigma}$ given by \req{isigma}, and the fact that the class of $w$ is the same as the class of $BwB^{-1}$ in $G/\Gamma_3(G)$ because $B\in \Gamma_{2}(G)$.
Part~(\ref{it:gamma3c}) follows from the presentation of $P_{2}(S_{g})$ given in Theorem~\ref{th:fadhu}, using once more the fact that $B\in \Gamma_{2}(G)$. 
\end{proof}

\begin{prop}\label{prop:maineq} The projection of equation~(\ref{eq:comm}) onto $G/\Gamma_3(G)$ is given by: 
\begin{equation}\label{eq:prodcomm}
\begin{aligned}
B =& \left[\widetilde v_1, \widetilde w_2\right]
\left[\widetilde w_1, \widetilde v_2^{-1}\right]
 \left(\widetilde \rho_1\rho_2\widetilde v_1 \widetilde w_1\widetilde
v_1^{-1} w_2^{-1}\right)\left(\widetilde \rho_2 \rho_1\widetilde
v_2 \widetilde w_2\widetilde v_2^{-1} w_1^{-1}\right).
\end{aligned}
\end{equation}
\end{prop}

\begin{proof} 
First note by Theorem~\ref{th:fadhu} that $G_{\text{Ab}}=(P_{1}(S_{g}))_{\text{Ab}}\times (P_{1}(S_{g}))_{\text{Ab}}\cong \Z^{2g} \times \Z^{2g}$, where a basis of the first (resp.\ second) $(P_{1}(S_{g}))_{\text{Ab}}$-factor consists of the images of $\rho_{1,i}$ (resp.\ $\rho_{2,i}$), for $i=1,\ldots, 2g$. 
The element $\widetilde \rho_{1} \rho_{2}$ of $G$ belongs to $\mathbb{F}_{2}$, and so $\left\lvert \widetilde \rho_{1} \rho_{2} \right\rvert_{\rho_{1,i}}=0$ for all $1\leq i\leq 2g$. Further,
\begin{equation*}
\left\lvert \widetilde \rho_{1} \rho_{2} \right\rvert_{\rho_{2,i}}= \left\lvert \widetilde \rho_{1}  \right\rvert_{\rho_{2,i}}+ \left\lvert \rho_{2} \right\rvert_{\rho_{2,i}}=\left\lvert \rho_{1}  \right\rvert_{\rho_{1,i}}+ \left\lvert \rho_{2} \right\rvert_{\rho_{2,i}}=0
\end{equation*}
by \req{expsumzero}. Thus $\left\lvert \widetilde \rho_{1} \rho_{2} \right\rvert_{\rho_{k,i}}=0$ for all $k\in\brak{1,2}$ and $i\in \brak{1,\ldots,2g}$. This implies that $\widetilde \rho_{1} \rho_{2}\in \Gamma_{2}(G)$. A similar argument shows that $\widetilde \rho_{2} \rho_{1}\in \Gamma_{2}(G)$.

We now take equation~(\ref{eq:comm}) modulo $\Gamma_{3}(G)$. Since $\widetilde \rho_{1} \rho_{2}, B, \widetilde \rho_{2} \rho_{1}\in \Gamma_{2}(G)$, and using the fact that $\Gamma_{2}(G)/\Gamma_{3}(G)$ is central in $G/\Gamma_{3}(G)$ as well as Lemma~\ref{lem:commu1}(\ref{it:gamma3a}), we obtain
\begin{align}
B &= \left[\widetilde v_1, \widetilde v_2 \right] \left[
\widetilde v_1\widetilde w_1, \widetilde w_2\widetilde
v_2^{-1}\right] \left[\widetilde v_2 \widetilde w_2\widetilde
v_2^{-1},\widetilde v_1\widetilde w_1\widetilde
v_1^{-1} w_2^{-1} \right] \left(\widetilde \rho_1\rho_2\widetilde v_1 \widetilde w_1\widetilde
v_1^{-1} w_2^{-1}\right)\left(\widetilde \rho_2 \rho_1 \widetilde
v_2 \widetilde w_2\widetilde v_2^{-1} w_1^{-1}\right)\notag\\
& = \left[\widetilde v_1, \widetilde v_2 \right] \left[
\widetilde v_1\widetilde w_1, \widetilde w_2\widetilde
v_2^{-1}\right] \left[\widetilde w_2,\widetilde w_1 w_2^{-1} \right] \left(\widetilde \rho_1\rho_2\widetilde v_1 \widetilde w_1\widetilde
v_1^{-1} w_2^{-1}\right)\left(\widetilde \rho_2 \rho_1 \widetilde
v_2 \widetilde w_2\widetilde v_2^{-1} w_1^{-1}\right) \notag\\
& = \left[\widetilde v_1, \widetilde v_2 \right] 
\left[
\widetilde v_1, \widetilde w_2\right] 
\left[
\widetilde v_1, \widetilde
v_2^{-1}\right] 
\left[
\widetilde w_1, \widetilde
v_2^{-1}\right] 
\left[\widetilde w_2,w_2^{-1} \right] 
\left(\widetilde \rho_1\rho_2\widetilde v_1 \widetilde w_1\widetilde
v_1^{-1} w_2^{-1}\right)\left(\widetilde \rho_2 \rho_1 \widetilde
v_2 \widetilde w_2\widetilde v_2^{-1} w_1^{-1}\right) \notag\\
& = 
\left[
\widetilde v_1, \widetilde w_2\right]  
\left[
\widetilde w_1, \widetilde
v_2^{-1}\right] 
\left[\widetilde w_2,w_2^{-1} \right] 
\left(\widetilde \rho_1\rho_2\widetilde v_1 \widetilde w_1\widetilde
v_1^{-1} w_2^{-1}\right)\left(\widetilde \rho_2 \rho_1 \widetilde
v_2 \widetilde w_2\widetilde v_2^{-1} w_1^{-1}\right)\label{eq:nearprodcomm}
\end{align}
in $G/\Gamma_{3}(G)$. Using Lemma~\ref{lem:commu1}(\ref{it:gamma3c}), we see that in $G/\Gamma_{3}(G)$, the only non-trivial contributions in $[\widetilde w_2,w_2^{-1}]$ come from terms of the form 
$[\rho_{1,i},\rho_{2,i+1}]=B$ and $[\rho_{1,i+1},\rho_{2,i}]=B^{-1}$ for $i$ odd. Thus in $G/\Gamma_{3}(G)$, the $B$-coefficient of $[\widetilde w_2,w_2^{-1}]$ is given by:
\begin{multline*}
-\sum_{\substack{1\leq i\leq 2g\\ \text{$i$ odd}}} \left\lvert \widetilde w_2 \right\rvert_{\rho_{1,i}} \left\lvert w_2 \right\rvert_{\rho_{2,i+1}}+ \sum_{\substack{1\leq i\leq 2g\\ \text{$i$ odd}}} \left\lvert \widetilde w_2 \right\rvert_{\rho_{1,i+1}} \left\lvert w_2 \right\rvert_{\rho_{2,i}}= 
-\sum_{\substack{1\leq i\leq 2g\\ \text{$i$ odd}}} \left\lvert w_2 \right\rvert_{\rho_{2,i}} \left\lvert w_2 \right\rvert_{\rho_{2,i+1}}+\\ \sum_{\substack{1\leq i\leq 2g\\ \text{$i$ odd}}} \left\lvert w_2 \right\rvert_{\rho_{2,i+1}} \left\lvert w_2 \right\rvert_{\rho_{2,i}}=0.
\end{multline*}
Hence $[\widetilde w_2,w_2^{-1}]=1$ in $G/\Gamma_{3}(G)$, and \req{nearprodcomm} thus reduces to \req{prodcomm}.
\end{proof}

\begin{rem}\label{rem:summar}
We summarise some properties of the factors of equation~(\ref{eq:prodcomm}):
\begin{enumerate}[(a)]
\item The factors  $ [\widetilde v_1, \widetilde w_2], [\widetilde w_1, \widetilde v_2^{-1}]$ belong to $N$ because $\widetilde v_1, \widetilde w_1 \in \mathbb   F_2$ 
and $ \widetilde w_2, \widetilde v_2^{-1} \in \mathbb F_1$.
\item The factors  $\widetilde \rho_1\rho_2\widetilde v_1 \widetilde w_1\widetilde
v_1^{-1} w_2^{-1}$ and $\widetilde \rho_2 \rho_1\widetilde
v_2 \widetilde w_2\widetilde v_2^{-1} w_1^{-1}$ belong to $N$ since their images in $P_1(S_g)\times P_1(S_g)$ belong to the subgroups $P_1(S_g)\times \brak{1}$,
$\brak{1}\times P_1(S_g)$ respectively, and $B$ projects to the trivial element.
\item\label{it:summarc} The elements  $(\widetilde \rho_1\rho_2), [\widetilde v_1, \widetilde w_1], \widetilde
w_1 w_2^{-1}$ belong to $\mathbb F_2\cap \Gamma_2(G)$, and 
$ \widetilde \rho_2\rho_1, [\widetilde v_2, \widetilde w_2], \widetilde w_2 w_1^{-1}$ belong to $\mathbb F_1\cap \Gamma_2(G)$. 
\end{enumerate}
\end{rem}

We now compute the group $\Gamma_2(G)/\Gamma_3(G)$.

\begin{prop}\label{prop:central}\mbox{}
\begin{enumerate}[(a)]
\item\label{it:part1} The group $\Gamma_2(G)/\Gamma_3(G)$ is free Abelian of rank $2g(2g-1)-1$; a basis is given by the classes of the elements of $\setl{e_{k,i,j}, B}{\text{$k=1,2,\,1 \leq i<j\leq 2g$ and $i\neq 2g-1$}}$, where $e_{k,i,j}=[\rho_{k,i}, \rho_{k,j}]$ for all $k=1,2$ and $1 \leq i<j\leq 2g$.

\item\label{it:part2} Given $v,w\in P_2(S_g)$, the commutator $[v,w]$, considered as an element of $G/\Gamma_3(G)$, belongs to $\Gamma_2(G)/\Gamma_3(G)$, and
\begin{enumerate}[(i)]

\item\label{it:part2i} $\left\lvert[v,w]\right\rvert_{e_{k,i,j}}=d_{k,i,j}(v,w)$ for $k=1,2$,  $1\leq i<j\leq 2g$  and   $(i, j)\ne (2t-1,2t)$ for all $1\leq t \leq g$,
\item\label{it:part2ii}  $\left\lvert[v,w]\right\rvert_{e_{k,2i-1,2i}}=d_{k,2i-1,2i}(v,w)-d_{k,2g-1,2g}(v,w)$  for all $k=1,2$ and $1 \leq i<g$,

\item\label{it:part2iii} $\displaystyle \left\lvert[v,w]\right\rvert_{B}=-d_{1,2g-1,2g}(v,w)-d_{2,2g-1,2g}(v,w)+ \sum_{1\leq i\leq g} a_{2i-1,2i}(v,w)$,
\end{enumerate}
where $\displaystyle \left\lvert u\right\rvert_{B}$ and $\displaystyle \left\lvert u\right\rvert_{e_{k,i,j}}$ denote the exponent sum of the element $u\in \Gamma_2(G)/\Gamma_3(G)$ with respect to the basis elements of part~(\ref{it:part1}), and where
\begin{equation*}
d_{k,i,j}(v,w)= 
\begin{vmatrix}
\left\lvert v\right\rvert_{\rho_{k,i}}& \left\lvert v\right\rvert_{\rho_{k,j}}\\
\left\lvert w\right\rvert_{\rho_{k,i}}& \left\lvert w\right\rvert_{\rho_{k,j}}
\end{vmatrix}\;\text{and} \; 
a_{2i-1,2i}(v,w)= 
\begin{vmatrix}
\left\lvert v\right\rvert_{\rho_{2, 2i-1}}& \left\lvert v\right\rvert_{\rho_{2, 2i}}\\
\left\lvert w\right\rvert_{\rho_{1, 2i-1}} & \left\lvert w\right\rvert_{\rho_{1, 2i}}
\end{vmatrix}+
\begin{vmatrix}
\left\lvert v\right\rvert_{\rho_{1, 2i-1}} & \left\lvert v\right\rvert_{\rho_{1, 2i}}\\
\left\lvert w\right\rvert_{\rho_{2, 2i-1}} & \left\lvert w\right\rvert_{\rho_{2, 2i}} 
\end{vmatrix}.
\end{equation*}
\end{enumerate}
\end{prop}





\begin{proof}\mbox{}
\begin{enumerate}[(a)]
\item Let $G_{12}$ denote the group defined by a presentation with generating set 
\begin{equation*}
\setl{a_{k,1},\ldots,a_{k,2g}, b_{k,i,j},  \beta}{k=1,2,\; 1 \leq i<j\leq 2g, \, i\neq 2g-1 },
\end{equation*}
and defining relations:
\begin{enumerate}[(I)]
\item\label{it:relIG12} $ b_{k,i,j}=[a_{k,i}, a_{k,j}]$ for $k=1,2$, $1\leq i<j\leq 2g$, where $i\neq 2g-1$.
\item $\beta =[a_{k,1},a_{k,2}^{-1}]\cdots [a_{k,2g-1},a_{k,2g}^{-1}]=[a_{2,2i-1},a_{1,2i}]=[a_{1,2i-1},a_{2,2i}]$  for all $k\in\brak{1,2}$ and $1\leq i\leq g$. 
\item\label{it:relIIIG12} $[a_{1,i},a_{2,j}]=1$  for all $1\leq i,j\leq 2g$, where $\brak{i,j}\neq \brak{2t-1,2t}$ for all $1\leq t\leq g$.

\item\label{it:relIVG12} For $k=1,2$ and $1\leq i<j\leq 2g$, the elements $b_{k,i,j}$ and $\beta$ belong to the centre of the group $G_{12}$.
\end{enumerate}

We will construct a homomorphism from $G_{12}$ to $G/\Gamma_3(G)$ and conversely. To define a homomorphism from $G_{12}$ to $G/\Gamma_3(G)$, consider the map defined on the generators of $G_{12}$ by $\beta\longmapsto B$, $a_{k,l}\longmapsto \rho_{k,l}$, and $b_{k,i,j}\longmapsto e_{k,i,j}$ for all $k\in\brak{1,2}$, $1\leq l\leq 2g$ and $1 \leq i<j\leq 2g$. Using Theorem~\ref{th:fadhu} and Lemma~\ref{lem:commu1}, a straightforward calculation shows that the images of the relations of the presentation of $G_{12}$ are satisfied in the group $G/\Gamma_3(G)$, and thus we obtain a homomorphism from $G_{12}$ onto $G/\Gamma_3(G)$. Conversely, consider the map from $\map{\phi}{G}{G_{12}}$ defined on the generators of $G$ by $\rho_{k,j}\longmapsto a_{k,j}$ for all $k\in \brak{1,2}$ and $j\in \brak{1,\ldots,2g}$.
Since $[a_{k,2i-1},a_{k,2i}^{-1}]=a_{k,2i}^{-1}(a_{k,2i}a_{k,2i-1} a_{k,2i}^{-1}a_{k,2i-1}^{-1})a_{k,2i}=a_{k,2i}^{-1} b_{k,2i-1,2i}^{-1}a_{k,2i}=b_{k,2i-1,2i}^{-1}$ for all $k\in\brak{1,2}$ and $1\leq i\leq g$, we conclude from relations~(\ref{it:relIG12}) and~(\ref{it:relIVG12}) above that $\beta$ is central in $G_{12}$. 
Taking the image of relation~(I) of Theorem~\ref{th:fadhu} shows that $\phi(B)=\beta$, and applying $\phi$ to the remaining relations of $G$ and using these two facts about $\beta$, we conclude that $\phi$ extends to a homomorphism of $G$ onto $G_{12}$. Since $\beta$ and the $b_{k,i,j}$ belong to the centre of $G_{12}$, we see that $\Gamma_{2}(G_{12})$ is the Abelian group generated by the $b_{k,i,j}$, and that $\Gamma_3(G_{12})$ is trivial. It follows that $\phi$ factors through $G/\Gamma_3(G)$. Since $\phi([a_{k,i},a_{k,j}])=b_{k,i,j}$ for all $k\in \brak{1,2}$ and $1\leq i<j\leq 2g$, we thus obtain two homomorphisms between $G_{12}$ to $G/\Gamma_3(G)$, where one is the inverse of the other.  In particular, $G_{12}$ and $G/\Gamma_3(G)$ are isomorphic, and hence $\Gamma_{2}(G_{12})$ is isomorphic to $\Gamma_2(G)/\Gamma_3(G)$. By considering the Abelianisation of $G_{12}$, one may check using the relations~(I)--(IV) above that $\Gamma_{2}(G_{12})$ is a free Abelian subgroup of $G_{12}$ with basis $\setr{\beta,b_{k,i,j}}{k=1,2,\; 1\leq i<j\leq 2g,\; i\neq 2g-1}$, and this proves part~(\ref{it:part1}).

\item Let $v,w\in G$, and consider their classes modulo $\Gamma_{3}(G)$, which we also denote by $v,w$ respectively. Then in $G/\Gamma_{3}(G)$, we have 
\begin{equation}\label{eq:vwexpr}
v=\left( 
\prod_{k=1}^2 \left( 
\prod_{i=1}^{2g} \rho_{k,i}^{\left\lvert v \right\rvert_{\rho_{k,i}}}
\right)
\right) \ldotp v' \quad \text{and} \quad w=\left( 
\prod_{k=1}^2 \left( 
\prod_{i=1}^{2g} \rho_{k,i}^{\left\lvert w \right\rvert_{\rho_{k,i}}}
\right)
\right) \ldotp w',
\end{equation}
where $v',w'\in \Gamma_{2}(G)/\Gamma_{3}(G)$. We now calculate the coefficients of $[v,w]$ in the given basis of $\Gamma_{2}(G)/\Gamma_{3}(G)$, noting that $v',w'$ may be ignored since they are central in $G/\Gamma_{3}(G)$. From Lemma~\ref{lem:commu1} and part~(\ref{it:part1}), if $1\leq i<j\leq 2g$ and $k,l\in\brak{1,2}$, we have that
\begin{equation}\label{eq:useful}
[\rho_{k,i}^s, \rho_{l,j}^t]=
\begin{cases}
e_{k,i,j}^{st} & \text{if $k=l$}\\
B^{st} & \text{if $k\neq l$ and $(i,j)=(2t-1,2t)$ for some $t\in \brak{1,\ldots,g}$}\\
1 & \text{if $k\neq l$ and $(i,j)\neq(2t-1,2t)$ for all $t\in \brak{1,\ldots,g}$,}
\end{cases}
\end{equation}
and from relation~(I) of Theorem~\ref{th:fadhu} and Lemma~\ref{lem:commu1}, we have 
\begin{equation}\label{eq:2geq}
[\rho_{k,2g-1},\rho_{k,2g}]=e_{k,1,2}^{-1} \cdots e_{k,2g-3,2g-2}^{-1} B^{-1}.
\end{equation}
Thus if $(i,j)\neq (2t-1,2t)$ for all $t\in \brak{1,\ldots,g}$, we obtain
\begin{align*}
\left\lvert [v,w] \right\rvert_{e_{k,i,j}}=\left\lvert v \right\rvert_{\rho_{k,i}} \left\lvert w \right\rvert_{\rho_{k,j}}-\left\lvert v \right\rvert_{\rho_{k,j}} \left\lvert w \right\rvert_{\rho_{k,i}}=d_{k,i,j}(v,w)
\end{align*}
obtained from the coefficients of $\rho_{k,i}$ and $\rho_{k,j}$ in \req{vwexpr} which gives~(\ref{it:part2i}), while if $i\in \brak{1,\ldots,g-1}$, we obtain an extra term in the expression for the coefficient of $e_{k,2i-1,2i}$ from the coefficients of $\rho_{k,2g-1}$ and $\rho_{k,2g}$ via \req{2geq}, and so
\begin{align*}
\left\lvert [v,w] \right\rvert_{e_{k,2i-1,2i}}=d_{k,2i-1,2i}(v,w)-d_{k,2g-1,2g}(v,w),
\end{align*}
which gives~(\ref{it:part2ii}). Finally, the $B$-coefficient of $[v,w]$ is obtained from three different types of expression: the first emanates from the coefficients of $\rho_{1,2i-1}$ and $\rho_{2,2i}$ for each $1\leq i\leq g$, which gives rise to a coefficient
\begin{equation*}
\begin{vmatrix}
\left\lvert v\right\rvert_{\rho_{1, 2i-1}}& \left\lvert v\right\rvert_{\rho_{2, 2i}}\\
\left\lvert w\right\rvert_{\rho_{1, 2i-1}} & \left\lvert w\right\rvert_{\rho_{2, 2i}}
\end{vmatrix},
\end{equation*}
the second comes from the coefficients of $\rho_{2,2i-1}$ and $\rho_{1,2i}$ for each $1\leq i\leq g$, which gives rise to a coefficient 
\begin{equation*}
\begin{vmatrix}
\left\lvert v\right\rvert_{\rho_{2, 2i-1}}& \left\lvert v\right\rvert_{\rho_{1, 2i}}\\
\left\lvert w\right\rvert_{\rho_{2, 2i-1}} & \left\lvert w\right\rvert_{\rho_{1, 2i}}
\end{vmatrix},
\end{equation*}
and the third is given by the coefficient of $e_{k,2g-1,2g}$ via \req{2geq} for $k\in \brak{1,2}$, which yields a coefficient $-d_{1,2g-1,2g}(v,w)-d_{2,2g-1,2g}(v,w)$. The sum of the first and second coefficients is equal to $a_{2i-1,2i}(v,w)$. Taking the sum of all of these coefficients leads to $\left\lvert [v,w] \right\rvert_{B}$ given in~(\ref{it:part2ii}), and this completes the proof of the proposition. \qedhere
%
\end{enumerate}
\end{proof}

Using Proposition~\ref{prop:central}, we are now in a position to prove 
Theorem~\ref{case5},
 which will follow easily from \repr{solequat}. Consider the quotient of $\Gamma_2(G)/\Gamma_3(G)$ obtained by identifying $e_{1,i,j}$ with $e_{2,i,j}$ for all $1\leq i<j\leq 2g$ and $i\neq 2g-1$. We denote this quotient by $Q$, and the image of $e_{1,i,j}$ and $e_{2,i,j}$ in $Q$ by $e_{i,j}$. By Proposition~\ref{prop:central}, the group $\Gamma_2(G)/\Gamma_3(G)$ is the direct sum of three free Abelian subgroups $H$, $\ang{B}$ and $L$, where $\set{ [\rho_{k,2i-1}, \rho_{k,2i}]}{k=1,2,\, 1 \leq i< g}$ is a basis of $H$, $\brak{B}$ is a basis of $\ang{B}$, and 
\begin{equation*}
\set{ [\rho_{k,i}, \rho_{k,j}]}{\text{$k=1,2$,  $1 \leq i<j\leq 2g$ and $(i,j)\neq (2t-1, 2t)$ for all $t\in \brak{1,\ldots,g}$}}
\end{equation*}
is a basis of $L$. Moreover, $H$ (resp.\ $L$) is the direct sum $H_1\oplus H_2$ (resp.\ $L_1\oplus L_2$) where for $k=1,2$,
$\set{[\rho_{k,2i-1}, \rho_{k,2i}]}{1\leq  i< g}$ is a basis of $H_k$, and 
\begin{equation*}
\set{ [\rho_{k,i}, \rho_{k,j}]}{\text{$1 \leq i<j\leq 2g$ and $(i,j)\neq (2t-1, 2t)$ for all $t\in \brak{1,\ldots,g}$}}
\end{equation*}
is a basis of $L_k$. Observe that the image of $H_1$ (resp.\ $L_1$) in $Q$ coincides with the image of $H_2$ (resp.\ $L_2$). Let $\overline Q=Q\otimes \mathbb{Z}_2$, and let $\overline B$, $\overline H$ and $\overline L$ denote the projection of $B$, $H$ and $L$ respectively in $\overline Q$.   

\begin{prop}\label{prop:solequat} 
Equation~(\ref{eq:main1}) has no solution in $P_{2}(S_{g})$.
\end{prop}

\begin{proof} 
We saw previously that equation~(\ref{eq:main1}) is equivalent in turn to \req{main2}, and to \req{comm}, and that its projection onto $G/\Gamma_{3}(G)$ is given by \req{prodcomm}. So 
to show that equation~(\ref{eq:main1}) has no solution in $P_{2}(S_{g})$ it suffices to show that the projection of \req{prodcomm} onto the group $\overline Q$ has no solution.
Now $\overline H, \ang{\overline{B}}$ and $\overline L$ are $\mathbb{Z}_2$-vector spaces of dimension equal to half the rank of $H$ 
(as a free Abelian group), $1$, and half the rank of $L$ (as a free Abelian group) respectively, and we have a decomposition of $\overline Q$ as $\overline H\oplus \ang{\overline{B}}\oplus \overline L$. We have that $\Gamma_2(\overline Q)$ is isomorphic to a sum of $\Z_2$'s; a basis is given by the set
$\set{\overline e_{i,j}, \overline B}{1\leq i<j\leq 2g,\, i\neq 2g-1}$, where $\overline e_{i,j}$ denotes the projection (from $Q$ to $\overline Q$) of $e_{i,j}$.
From now on we study the projection of equation~(\ref{eq:prodcomm}) onto $\overline Q$ (apart from the basis elements of $\overline Q$, notationally we do not distinguish between elements of $\Gamma_2(G)/\Gamma_3(G)$ and their projection into $\overline Q$):
\begin{equation}\label{eq:prodcomm1}
\overline B=[\widetilde v_1, \widetilde w_2]
[\widetilde w_1, \widetilde v_2^{-1}]
(\widetilde \rho_1\rho_2 )   [ \widetilde v_1, \widetilde w_1](\widetilde w_1  w_2^{-1})(\widetilde \rho_2 \rho_1 )[\widetilde v_2,
 \widetilde w_2](\widetilde w_2 w_1^{-1}),
\end{equation}
where each of the factors belongs to $\Gamma_2(\overline Q)$, using \rerem{summar}(\ref{it:summarc}), and so commute pairwise. 
We now examine the various terms appearing in \req{prodcomm1}. 

\begin{enumerate}[(a)]
\item We have $(\widetilde \rho_1\rho_2)(\widetilde \rho_2\rho_1)=
[\widetilde \rho_1,\rho_2]( \rho_2\widetilde \rho_1)(\widetilde \rho_2\rho_1)$.  We claim that $[\widetilde \rho_1,\rho_2]=1$ in $\Gamma_{2}(G)/\Gamma_{3}(G)$, and so in $\Gamma_{2}(\overline Q)$. To prove the claim, we calculate the coefficients of $[\widetilde \rho_1,\rho_2]$ on the basis of $\Gamma_{2}(G)/\Gamma_{3}(G)$ using \repr{central}(\ref{it:part2}). First recall that $\widetilde \rho_1,\rho_2\in \mathbb{F}_{2}$, so
\begin{equation}\label{eq:rho1tilde}
\left\lvert \widetilde \rho_1 \right\rvert_{\rho_{1,i}}=\left\lvert \rho_2 \right\rvert_{\rho_{1,i}}=0 \quad \text{for all $1\leq i\leq 2g$}
\end{equation}
and hence $d_{1,i,j}(\widetilde \rho_1,\rho_2)=\left\lvert [\widetilde \rho_1,\rho_2]\right\rvert_{e_{1,i,j}}=0$ for all $1\leq i<j\leq 2g$ by \repr{central}(\ref{it:part2i}) and~(\ref{it:part2ii}). Further,
\begin{align*}
d_{2,i,j}(\widetilde \rho_1,\rho_2)&=
\left\lvert \widetilde \rho_1 \right\rvert_{\rho_{2,i}} \left\lvert \rho_2 \right\rvert_{\rho_{2,j}}-\left\lvert \widetilde \rho_1 \right\rvert_{\rho_{2,j}} \left\lvert \rho_2 \right\rvert_{\rho_{2,i}}=
\left\lvert \rho_1 \right\rvert_{\rho_{1,i}} \left\lvert \rho_2 \right\rvert_{\rho_{2,j}}-\left\lvert \rho_1 \right\rvert_{\rho_{1,j}} \left\lvert \rho_2 \right\rvert_{\rho_{2,i}}\\
&=
-\left\lvert \rho_2 \right\rvert_{\rho_{2,i}} \left\lvert \rho_2 \right\rvert_{\rho_{2,j}}+\left\lvert \rho_2 \right\rvert_{\rho_{2,j}} \left\lvert \rho_2 \right\rvert_{\rho_{2,i}} \quad\text{by \req{expsumzero}}\\
&=0.
\end{align*}
Finally, 
\begin{align*}
a_{2i-1,2i}(\widetilde \rho_1,\rho_2)&= \left\lvert \widetilde \rho_1 \right\rvert_{\rho_{2,2i-1}} \left\lvert \rho_2 \right\rvert_{\rho_{1,2i}}-\left\lvert \widetilde \rho_1 \right\rvert_{\rho_{2,2i}} \left\lvert \rho_2 \right\rvert_{\rho_{1,2i-1}}+
\left\lvert \widetilde \rho_1 \right\rvert_{\rho_{1,2i-1}} \left\lvert \rho_2 \right\rvert_{\rho_{2,2i}}-\left\lvert \widetilde \rho_1 \right\rvert_{\rho_{1,2i}} \left\lvert \rho_2 \right\rvert_{\rho_{2,2i-1}}\\
&=0,
\end{align*}
using \req{rho1tilde}. So $\left\lvert [\widetilde \rho_1,\rho_2]\right\rvert_{B}=0$, and we conclude that $[\widetilde \rho_1,\rho_2]=1$ in $\Gamma_{2}(G)/\Gamma_{3}(G)$, which proves the claim. 

\item\label{it:rho2tilde} Consider the terms $\rho_2\widetilde \rho_1$ and $\widetilde \rho_2\rho_1$. As an element of $G$, we have that $\rho_2\widetilde \rho_1\in \mathbb{F}_{2}$, and so $\iota_{\sigma}(\rho_2\widetilde \rho_1)=B \widetilde \rho_2 \rho_1 B^{-1}$ by \req{isigma}. Since $\Gamma_{3}(G)$ is characteristic in $G$, $\iota_{\sigma}$ induces an automorphism of $G/\Gamma_{3}(G)$ which we also denote by $\iota_{\sigma}$. But $B\in \Gamma_{2}(G)$, so $\iota_{\sigma}(\rho_2\widetilde \rho_1)=\widetilde \rho_2 \rho_1$ in $\Gamma_{2}(G)/\Gamma_{3}(G)$. Now 
\begin{equation}\label{eq:iotaekij}
\iota_{\sigma}(e_{k,i,j})=e_{k',i,j} \quad \text{for all $1\leq i<j\leq 2g$ and $k,k'\in {1,2}$, where $k\neq k'$.} 
\end{equation}
So $\left\lvert \rho_2\widetilde \rho_1\widetilde \rho_2\rho_1\right\rvert_{B}=\left\lvert \iota_{\sigma}(\widetilde \rho_2\rho_1)\widetilde \rho_2\rho_1\right\rvert_{B}$, and since $\iota_{\sigma}(B)=B$, it follows that $\left\lvert \rho_2\widetilde \rho_1\widetilde \rho_2\rho_1\right\rvert_{B}$ is even. Hence the $\overline B$-coefficient of $\rho_2\widetilde \rho_1\widetilde \rho_2\rho_1$ is zero in $\overline Q$. Using \req{iotaekij}, we see that
\begin{equation*}
\left\lvert \rho_2\widetilde \rho_1\widetilde \rho_2\rho_1\right\rvert_{e_{1,i,j}}= \left\lvert \rho_2\widetilde \rho_1\widetilde \rho_2\rho_1\right\rvert_{e_{2,i,j}} \quad \text{for all $1\leq i<j\leq 2g$,}
\end{equation*}
hence the $\overline{e}_{i,j}$-coefficient of $\rho_2\widetilde \rho_1\widetilde \rho_2\rho_1$ is also zero in $\overline Q$, and thus $\rho_2\widetilde \rho_1\widetilde \rho_2\rho_1$ is trivial in $\overline Q$.

\item Now consider $\widetilde w_1 w_2^{-1}$ and $\widetilde w_2 w_1^{-1}$. We have
$\widetilde w_1 w_2^{-1}\widetilde w_2 w_1^{-1}=(\widetilde w_1 w_2^{-1})^2 w_2 \widetilde w_1^{-1} \widetilde w_2  w_1^{-1}$. Since it is a square, $(\widetilde w_1 w_2^{-1})^2$ is certainly trivial in $\overline Q$. As in case~(\ref{it:rho2tilde}) above, $w_2 \widetilde w_1^{-1} \widetilde w_2  w_1^{-1}$ is also trivial in $\overline Q$.
\end{enumerate}
Hence \req{prodcomm1} reduces to $\overline B=[\widetilde v_1, \widetilde w_2]
[\widetilde w_1, \widetilde v_2^{-1}] [ \widetilde v_1, \widetilde w_1][\widetilde v_2,
 \widetilde w_2]$ in $\overline Q$. Using the results of \relem{commu1}, we can rewrite this as
\begin{equation}\label{eq:finalqbar}
\overline B=[\widetilde v_1 \widetilde v_2, \widetilde w_1\widetilde w_2].
\end{equation}
First suppose that $g=1$. In this case, the basis of $\Gamma_{2}(G)/\Gamma_{3}(G)$ is reduced to $\brak{B}$. Since $\widetilde v_2, \widetilde w_2\in \mathbb{F}_{1}$ and $\widetilde v_1, \widetilde w_1\in \mathbb{F}_{2}$, and using \repr{central}(\ref{it:part2iii}) and \req{expsumzero}, in $\Gamma_{2}(G)/\Gamma_{3}(G)$ we have
\begin{align*}
\left\lvert [\widetilde v_1 \widetilde v_2, \widetilde w_1\widetilde w_2]\right\rvert_{B}=& d_{1,1,2}(\widetilde v_1 \widetilde v_2, \widetilde w_1\widetilde w_2)+d_{2,1,2}(\widetilde v_1 \widetilde v_2, \widetilde w_1\widetilde w_2)+ c_{1,2}(\widetilde v_1 \widetilde v_2, \widetilde w_1\widetilde w_2)\\
=&
\left\lvert \widetilde v_{2} \right\rvert_{\rho_{1,1}} \left\lvert \widetilde w_{2} \right\rvert_{\rho_{1,2}}-
\left\lvert \widetilde v_{2} \right\rvert_{\rho_{1,2}} \left\lvert \widetilde w_{2} \right\rvert_{\rho_{1,1}}+
\left\lvert \widetilde v_{1} \right\rvert_{\rho_{2,1}} \left\lvert \widetilde w_{1} \right\rvert_{\rho_{2,2}}-
\left\lvert \widetilde v_{1} \right\rvert_{\rho_{2,2}} \left\lvert \widetilde w_{1} \right\rvert_{\rho_{2,1}}+\\
& \left\lvert \widetilde v_{2} \right\rvert_{\rho_{1,1}} \left\lvert \widetilde w_{1} \right\rvert_{\rho_{2,2}}-
\left\lvert \widetilde v_{2} \right\rvert_{\rho_{1,2}} \left\lvert \widetilde w_{1} \right\rvert_{\rho_{2,1}}+
\left\lvert \widetilde v_{1} \right\rvert_{\rho_{2,1}} \left\lvert \widetilde w_{2} \right\rvert_{\rho_{1,2}}-
\left\lvert \widetilde v_{1} \right\rvert_{\rho_{2,2}} \left\lvert \widetilde w_{2} \right\rvert_{\rho_{1,1}}\\
=&
\left\lvert v_{2} \right\rvert_{\rho_{2,1}} \left\lvert w_{2} \right\rvert_{\rho_{2,2}}-
\left\lvert v_{2} \right\rvert_{\rho_{2,2}} \left\lvert w_{2} \right\rvert_{\rho_{2,1}}+
\left\lvert v_{1} \right\rvert_{\rho_{1,1}} \left\lvert w_{1} \right\rvert_{\rho_{1,2}}-
\left\lvert v_{1} \right\rvert_{\rho_{1,2}} \left\lvert w_{1} \right\rvert_{\rho_{1,1}}+\\
& \left\lvert v_{2} \right\rvert_{\rho_{2,1}} \left\lvert w_{1} \right\rvert_{\rho_{1,2}}-
\left\lvert v_{2} \right\rvert_{\rho_{2,2}} \left\lvert w_{1} \right\rvert_{\rho_{1,1}}+
\left\lvert v_{1} \right\rvert_{\rho_{1,1}} \left\lvert w_{2} \right\rvert_{\rho_{2,2}}-
\left\lvert v_{1} \right\rvert_{\rho_{1,2}} \left\lvert w_{2} \right\rvert_{\rho_{2,1}}\\
=&
2\left(\left\lvert v_{2} \right\rvert_{\rho_{2,1}} \left\lvert w_{1} \right\rvert_{\rho_{1,2}}-
\left\lvert v_{2} \right\rvert_{\rho_{2,2}} \left\lvert w_{1} \right\rvert_{\rho_{1,1}}+
\left\lvert v_{1} \right\rvert_{\rho_{1,1}} \left\lvert w_{1} \right\rvert_{\rho_{1,2}}-
\left\lvert v_{1} \right\rvert_{\rho_{1,2}} \left\lvert w_{1} \right\rvert_{\rho_{1,1}}\right).
\end{align*}
Thus $[\widetilde v_1 \widetilde v_2, \widetilde w_1\widetilde w_2]$ is trivial in $\overline Q$, which contradicts \req{finalqbar}. So let us suppose that $g>1$. We will derive some restrictions on the element $w_1$ by studying \req{finalqbar} after projecting onto $\overline Q$.  For $i=1,\ldots, 2g$, let $a_{i}=\left\lvert v_{1} \right\rvert_{\rho_{1,i}}$, $b_{i}=\left\lvert \widetilde v_{2} \right\rvert_{\rho_{1,i}}=\left\lvert v_{2} \right\rvert_{\rho_{2,i}}$ and $c_{i}=\left\lvert w_{1} \right\rvert_{\rho_{1,i}}$, and let $d_{i}=a_{i}+b_{i}$.  The right-hand side of \req{finalqbar} may be written as a product of two types of term: $[\widetilde v_l,\widetilde w_m]$, where $l,m\in \brak{1,2}$ and $l\neq m$, and $[\widetilde v_l,\widetilde w_l]$, where $l\in \brak{1,2}$. In the first case, considered as an element of $\Gamma_{2}(G)/\Gamma_{3}(G)$, $[\widetilde v_l,\widetilde w_m]$ gives rise only to terms in $B$ by \req{useful}. In particular, in $\Gamma_{2}(G)/\Gamma_{3}(G)$ $\left\lvert [\widetilde v_l,\widetilde w_m]\right\rvert_{e_{k,i,j}}=0$ for all $k\in \brak{1,2}$ and $1\leq i<j\leq 2g$, and so the $\overline{e}_{i,j}$-coefficient of $[\widetilde v_l,\widetilde w_m]$, considered as an element of $\overline Q$, is zero. It follows from \req{finalqbar} that in $\overline Q$, the $\overline{e}_{i,j}$-coefficient of $[\widetilde v_1,\widetilde w_1][\widetilde v_2,\widetilde w_2]$ is zero for all $ (i,j)\neq (2t-1,2t)$ and $1\leq i<j\leq 2g$. 
But modulo $2$, this coefficient is also given by the sum 
\begin{equation}\label{eq:coeffviwi}
\begin{aligned}
\left\lvert [\widetilde v_2,\widetilde w_2] \right\rvert_{e_{1,i,j}}+ \left\lvert [\widetilde v_1,\widetilde w_1] \right\rvert_{e_{2,i,j}}&=\begin{vmatrix}
\left\lvert v_2\right\rvert_{\rho_{2,i}}& \left\lvert v_2\right\rvert_{\rho_{2,j}}\\
\left\lvert w_2\right\rvert_{\rho_{2,i}}& \left\lvert w_2\right\rvert_{\rho_{2,j}}
\end{vmatrix}+
\begin{vmatrix}
\left\lvert v_1\right\rvert_{\rho_{1,i}}& \left\lvert v_1\right\rvert_{\rho_{1,j}}\\
\left\lvert w_1\right\rvert_{\rho_{1,i}}& \left\lvert w_1\right\rvert_{\rho_{1,j}}
\end{vmatrix}\\
&= \begin{vmatrix}
\left\lvert v_1\right\rvert_{\rho_{1,i}}+ \left\lvert v_2\right\rvert_{\rho_{2,i}}& \left\lvert v_1\right\rvert_{\rho_{1,j}}+ \left\lvert v_2\right\rvert_{\rho_{2,j}}\\
\left\lvert w_1\right\rvert_{\rho_{1,i}}& \left\lvert w_1\right\rvert_{\rho_{1,j}}
\end{vmatrix}= \begin{vmatrix}
d_{i} & d_{j}\\
c_{i} & c_{j}
\end{vmatrix},
\end{aligned}
\end{equation}
using \req{expsumzero}, so $\begin{vmatrix}
\overline{d_{i}} & \overline{d_{j}}\\
\overline{c_{i}} & \overline{c_{j}}
\end{vmatrix}=\overline{0} \pmod 2$.

Suppose that $\overline{c_{i}}=\overline{0} \pmod 2$ (so $c_{i}$ is even) for all $i=1,\ldots,2g$. Since $c_{i}=\left\lvert w_1\right\rvert_{\rho_{1,i}}= \left\lvert w_2\right\rvert_{\rho_{2,i}}$ by \req{expsumzero}, it follows from \repr{central}(\ref{it:part2}) that $d_{k,l,m}(\widetilde v_q,\widetilde w_q)$ is even for all $k,q\in\brak{1,2}$ and $1\leq l<m\leq 2g$. Hence in $\overline Q$, the $\overline{B}$-coefficient of $[\widetilde v_1,\widetilde w_1] [\widetilde v_2,\widetilde w_2]$ is zero, which contradicts 
\req{finalqbar}. Thus there exists $1\leq i \leq 2g$ such that $\overline{c_{i}}\neq \overline{0} \pmod 2$.


Using \repr{central}(\ref{it:part2}), a calculation similar to that of \req{coeffviwi} shows that the $\overline{e}_{1,2}$- (resp.\ $\overline B$-) coefficient of $[\widetilde v_1,\widetilde w_1] [\widetilde v_2,\widetilde w_2]$ is equal to $\begin{vmatrix}
\overline{d_{1}} & \overline{d_{2}}\\
\overline{c_{1}} & \overline{c_{2}}
\end{vmatrix}+\begin{vmatrix}
\overline{d_{2g-1}} & \overline{d_{2g}}\\
\overline{c_{2g-1}} & \overline{c_{2g}}
\end{vmatrix}$. By \req{finalqbar}, this coefficient is equal to $\overline{0}$ (resp.\ $\overline{1}$), so $\begin{vmatrix}
\overline{d_{1}} & \overline{d_{2}}\\
\overline{c_{1}} & \overline{c_{2}}
\end{vmatrix}=\overline{1}$. Hence there exists $l\in \brak{1,2}$ such that $\overline{c_{l}}\neq \overline{0}$. Now for all $m\in \brak{2g-1,2g}$, in $\overline Q$ the $\overline{e}_{l,m}$-coefficient of $[\widetilde v_1,\widetilde w_1] [\widetilde v_2,\widetilde w_2]$ is zero by \req{finalqbar}. By \req{coeffviwi}, this coefficient is equal to $\begin{vmatrix}
\overline{d_{l}} & \overline{d_{m}}\\
\overline{c_{l}} & \overline{c_{m}}
\end{vmatrix}$. Since $\overline{c_{l}}\neq \overline{0}$, this implies that $\begin{vmatrix}
\overline{d_{2g-1}} & \overline{d_{2g}}\\
\overline{c_{2g-1}} & \overline{c_{2g}}
\end{vmatrix}=\overline{0}$, but we know that this is the $\overline{B}$-coefficient in $\overline Q$ of $[\widetilde v_1,\widetilde w_1] [\widetilde v_2,\widetilde w_2]$. This contradicts \req{finalqbar}, and completes the proof of the proposition.
\end{proof}

\begin{proof}[Proof of Theorem~\ref{case5}.]
Consider the homomorphism $\map{\theta_{\tau}}{\pi_{1}(N_{3})}{\Z_{2}}$. Up to equivalence, we may suppose that $\theta_{\tau}$ satisfies one of the three conditions
(\ref{it:caseaaa})--(\ref{it:casecc}) given at the beginning of the discussion of this subcase (\ref{it:subcase4}).

In case~(a), we have $\theta_{\tau}(v)=\overline{0}$. We thus obtain a factorisation of diagram~(\ref{eq:basic}) as in Theorem~\ref{orien},
and so by Proposition~\ref{gen}, the Borsuk-Ulam 
property does not hold for the triple $(X,\tau,S_{g})$. In case~(b), we have
$\theta_{\tau}(v)=\overline{1}$ and  $\theta_{\tau}(a_1)=\theta_{\tau}(a_2)=\overline{0}$, and setting $\phi(v)=\sigma$, $\phi(a_{1})=\rho_{1,1}^{-1}$ and $\phi(a_{2})= \rho_{2,2}$ defines a factorisation of diagram~(\ref{eq:basic}) by the first relation of~(\ref{it:fadhuV}) of Theorem~\ref{th:fadhu}. Applying once more Proposition~\ref{gen}, we see that the Borsuk-Ulam property does not hold for the triple $(X,\tau,S_{g})$. 

Finally, consider case~(c), so $\theta_{\tau} (v)=\theta_{\tau}(a_2)=\overline{1}$ and $\theta_{\tau}(a_1)=\overline{0}$. It follows from Proposition~\ref{prop:solequat} that the non-existence of a solution to \req{main1} implies the non-existence of a solution of  \req{basic6}, and hence by Proposition~\ref{gen}, there is no factorisation of the diagram~(\ref{eq:basic}) by a homomorphism $\phi$. This completes the proof of the theorem.
\end{proof}

\section*{Appendix}

The purpose of this appendix is to reduce the number of cases to be analysed. The results presented here are known to the authors of~\cite{BGHZ}. For the benefit of the reader, we summarise these results and write them in a form that is more suitable for our purposes. Our problem is that of studying the existence of a solution to the algebraic factorisation problem presented in diagram~(\ref{eq:basic}) of Proposition~\ref{gen}. Using the notion of equivalence introduced at the end of \resec{GeBo}, our goal is to reduce the number of surjective homomorphisms $\map{\theta_{\tau}}{\pi_1(X/\tau)}{\Z_2}$ to be analysed, where $\pi_1(X/\tau)$ is isomorphic to the fundamental group of a compact, connected surface without boundary different from $\St[2]$ and $\rp$. We consider two cases, the first (resp.\ second) being that where the surface is orientable (resp.\ non-orientable). In the whole of this appendix, $X$ will be a finite-dimensional $CW$-complex equipped with a free cellular involution $\tau$.

\begin{prop}\label{prop:appendop}
Let $\pi_1(X/\tau)$ be isomorphic to the fundamental group of a compact, connected, orientable surface without boundary different from $\St[2]$ of genus $h$, and consider the presentation of $\pi_1(X/\tau)$ given by 
\begin{equation}\label{eq:prodcommh}
\setangr{a_1,a_2,\ldots,a_{2h-1},a_{2h}}{[a_1,a_2]\cdots[a_{2h-1},a_{2h}]\,}. 
\end{equation}
The existence of a solution to the algebraic factorisation problem of
diagram~(\ref{eq:basic}) of Proposition~\ref{gen} does not depend on the choice of surjective
homomorphism $\map{\theta_{\tau}}{\pi_1(X/\tau)}{P_2(S_g)}$.
In particular, it suffices to study the case  $\theta_{\tau}(a_1)=\overline{1}$ and  $\theta_{\tau}(a_i)=\overline{0}$ for all $1<i\leq 2h$.
\end{prop}

\begin{proof}
The following identities 
show that if $[a_1,a_2]\cdots[a_{2h-1},a_{2h}]$ is a product of commutators as in \req{prodcommh} where $\theta_{\tau}(a_{i})\neq \overline{0}$ for some $1\leq i\leq 2h$ then $\pi_1(X/\tau)$ admits a presentation
\begin{equation*}
\setangr{a'_1,a'_2,\ldots,a'_{2h-1},a'_{2h}}{[a'_1,a'_2]\cdots[a'_{2h-1},a'_{2h}]\,}, 
\end{equation*}
where 
\begin{multline}\label{eq:prodcommhprime}
[a_1,a_2]\cdots[a_{2h-1},a_{2h}]=[a'_1,a'_2]\cdots[a'_{2h-1},a'_{2h}], \\
\text{with $\theta_{\tau}(a'_{1})=\overline{1}$ and $\theta_{\tau}(a'_{i})=\overline{0}$ for all $1<i\leq 2h$.}
\end{multline}

\begin{enumerate}[(1)]
\item Let $(a^{\ast}, b^{\ast})=(a,ba)$. Then $[a,b]=[a^{\ast},b^{\ast}]$, and we may assume that
either $\theta_{\tau}(a^{\ast})$ or $\theta_{\tau}(b^{\ast})$ is zero.
\item Let $(a^{\ast}, b^{\ast})=(aba^{-1},a^{-1})$. Then $[a,b]=[a^{\ast},b^{\ast}]$, and we may assume
that  $\theta_{\tau}(a^{\ast})$ is zero.
\item Let $(a^{\ast}, b^{\ast}, c^{\ast}, d^{\ast})=([a,b]c[b,a], [a,b]d[b,a],  a,b)$. Then $[a,b]
[c,d]=[a^{\ast},b^{\ast}][c^{\ast}, d^{\ast}]$, and we may assume that there exists $1 \leq r \leq h$ such that
$\theta_{\tau}(a_i)$ is zero for $i\leq 2r$, and for $i>r$,
$\theta_{\tau}(a_{2i-1})= \overline{0}$ and $\theta_{\tau}(a_{2i})=
\overline{1}$.
\item Let $(a^{\ast}, b^{\ast}, c^{\ast}, d^{\ast})=(ac, c^{-1}bc,  c^{-1}bcb^{-1}c, dc^{-1}b^{-1}c)$.
Then $[a,b][c,d]=[a^{\ast},b^{\ast}][c^{\ast}, d^{\ast}]$ and if $\theta_{\tau}(a)=\theta_{\tau}(c)=\overline{0}$,
$\theta_{\tau}(b)=\theta_{\tau}(d)=\overline{1}$,
we obtain $\theta_{\tau}(a^{\ast})=\theta_{\tau}(c^{\ast})=\theta_{\tau}(d^{\ast})=\overline{0}$ and $\theta_{\tau}(b^{\ast})=\overline{1}$.
\end{enumerate}
Applying these four identities, we see that in order to 
analyse the algebraic factorisation problem for an arbitrary
surjective homomorphism $\theta_{\tau}$, it is sufficient to study the homomorphism $\theta_{\tau}$ given by  $\theta_{\tau}(a_{i})=\overline{1}$ if $i=1$, and $\overline{0}$ otherwise. This concludes the proof.
\end{proof}

\begin{rem}
From the above relations, in the orientable case, we deduce that any two surjective homomorphisms $\pi_1(X/\tau) \to P_2(S_g)$ are equivalent (in the sense given at the end of \resec{GeBo}).
\end{rem}

We now study the non-orientable case.

\begin{prop}\label{prop:appendor}
Suppose that $\pi_1(X/\tau)$ is isomorphic to the fundamental group of a compact, connected, non-orientable surface without boundary different from $\rp$ of genus $h\geq 2$.
\begin{enumerate}[(a)]
\item\label{it:caseaa} Let $h$ be odd, and consider the following presentation:
\begin{equation}\label{eq:orodd}
\pi_1(X/\tau)=\setangr{v,a_1,a_2,\ldots,a_{h-2},a_{h-1}}{v^2\cdot [a_1,a_2]\cdots[a_{h-2},a_{h-1}]}.
\end{equation}
In order to study the algebraic problem, it suffices to consider
the following three subcases:
\begin{enumerate}[(1)]
\item\label{it:casea1} $\theta_{\tau}(v)=\overline{0}$, $\theta_{\tau}(a_1)=\overline{1}$ and $\theta_{\tau}(a_i)=\overline{0}$ for all $i>1$.
\item\label{it:casea2} $\theta_{\tau}(v)=\overline{1}$, and $\theta_{\tau}(a_i)=\overline{0}$   for all $i\geq 1$.
\item\label{it:casea3} $\theta_{\tau}(v)=\overline{1}$, $\theta_{\tau}(a_1)=\overline{1}$ and $\theta_{\tau}(a_i)=\overline{0}$ for all $i>1$.
\end{enumerate}

\item Let $h$ be even, and consider the following presentation:
\begin{equation}\label{eq:oreven}
\pi_1(X/\tau)=\setangr{\alpha, \beta, a_1,a_2,\ldots,a_{2h-3},a_{2h-2}}{\alpha\beta\alpha\beta^{-1}[a_1,a_2] \cdots
[a_{2h-3},a_{2h-2}]}.
\end{equation}
\begin{enumerate}[(I)]
\item\label{it:casebI} If $h=2$ then in order to study the algebraic problem, it suffices to
consider the following subcases:
\begin{enumerate}[(1)]
\item\label{it:caseb1} $\theta_{\tau}(\alpha)=\overline{0}$ and $\theta_{\tau}(\beta)=\overline{1}$.
\item\label{it:caseb2}  $\theta_{\tau}(\alpha)=\overline{1}$ and $\theta_{\tau}(\beta)=\overline{0}$.
\end{enumerate}
\item\label{it:casebII} If $h\geq 4$ then in order to study the algebraic problem, it suffices to
consider the following subcases:
\begin{enumerate}[(1)]
\item\label{it:casec1} $\theta_{\tau}(\alpha)=\overline{0}$, $\theta_{\tau}(\beta)=\overline{1}$, and $\theta_{\tau}(a_i)=\overline{0}$   for all $i\geq
1$.
\item\label{it:casec2}  $\theta_{\tau}(\alpha)=\overline{0}$, $\theta_{\tau}(\beta)=\overline{0}$,  $\theta_{\tau}(a_1)=\overline{1}$ and
$\theta_{\tau}(a_i)=\overline{0}$ for all $i>1$.
\item\label{it:casec3}  $\theta_{\tau}(\alpha)=\overline{1}$, $\theta_{\tau}(\beta)=\overline{0}$, and  $\theta_{\tau}(a_i)=\overline{0}$   for all
$i\geq 1$.
\end{enumerate}
\end{enumerate}
\end{enumerate}
\end{prop}

\begin{proof}\mbox{}
\begin{enumerate}[(a)]
\item Let $h\geq 3$ be odd. Suppose first that $\theta_{\tau}(v)=\overline{0}$. By the relation of the presentation given by~\req{orodd}, we must have $\theta_{\tau}(a_i)=\overline{1}$ for some $i$. Using the transformations of the proof of \repr{appendop}, we may assume that $\theta_{\tau}(a_i)=\overline{1}$ if $i=1$ and zero if $i>1$, which is case~(\ref{it:casea1}). Now suppose that $\theta_{\tau}(v)=\overline{1}$. One possibility is that $\theta_{\tau}(a_i)=\overline{0}$ for all $i\geq 1$, which is case (\ref{it:casea2}). Now suppose that  for some $1\leq i\leq h-1$, we have $\theta_{\tau}(a_i)=\overline{1}$. Again using the transformations of the proof of \repr{appendop},  we
may assume that $\theta_{\tau}(a_i)=\overline{1}$ if $i=1$ and zero if $i>1$, which is case (\ref{it:casea3}). This completes the proof of part~(\ref{it:caseaa}).

\item  If $a,b\in \pi_{1}(X/\tau)$, let $[a,b]'=abab^ {-1}$ denote their twisted commutator. 

\begin{enumerate}[(I)]
\item Let $h=2$. Then there are three surjective  homomorphisms:
\begin{enumerate}[(i)]
\item $\theta_{\tau}(\alpha)=\overline{0}$ and $\theta_{\tau}(\beta)=\overline{1}$, which is case (\ref{it:caseb1}).
\item $\theta_{\tau}(\alpha)=\overline{1}$ and $\theta_{\tau}(\beta)=\overline{0}$, which is case (\ref{it:caseb2}).
\item $\theta_{\tau}(\alpha)=\theta_{\tau}(\beta)=\overline{1}$.
\end{enumerate}
Now if we let $(\alpha^{\ast},\beta^{\ast})=(\alpha, \beta\alpha)$, then we have $[\alpha,
\beta]'=[\alpha^{\ast}, \beta^{\ast}]'$. This shows that the second and third homomorphisms are equivalent, and this completes the proof of part~(\ref{it:casebI}).

\item Let $h\geq 4$. First we reduce the number of cases to five.  Arguing as in the case $h=2$ on the values of $\theta_{\tau}$ on $\alpha, \beta$, we see that we may reduce 
to the following cases:
\begin{enumerate}[(i)]
\item $\theta_{\tau}(\alpha)=\overline{0}=\theta_{\tau}(\beta)=\overline{0}$.
\item $\theta_{\tau}(\alpha)=\overline{0}$ and $\theta_{\tau}(\beta)=\overline{1}$.
\item $\theta_{\tau}(\alpha)=\overline{1}$ and $\theta_{\tau}(\beta)=\overline{0}$.
\end{enumerate}
For the first case
$\theta_{\tau}(\alpha)=\overline{0}=\theta_{\tau}(\beta)=\overline{0}$, we must have $\theta_{\tau}(a_i)=\overline{1}$ for some $1\leq i\leq 2h-2$. It then follows
from the proof of \repr{appendop} that we may assume that
$\theta_{\tau}(a_i)=\overline{1}$ if $i=1$ and zero if $i>1$. 
For the second case, $\theta_{\tau}(\alpha)=\overline{0}$ and $\theta_{\tau}(\beta)=\overline{1}$, we can either have $\theta_{\tau}(a_i)=\overline{0}$ for all $1\leq i\leq 2h-2$, or $\theta_{\tau}(a_i)=\overline{1}$ for some $1\leq i\leq 2h-2$. In the latter case, again
by the proof of \repr{appendop}, we may assume that
$\theta_{\tau}(a_i)=\overline{1}$ if $i=1$ and zero if $i>1$.
The third case $\theta_{\tau}(\alpha)=\overline{1}$ and $\theta_{\tau}(\beta)=\overline{0}$ is completely
analogous to the second case, and so the three cases above yield a total of five subcases:
\begin{enumerate}[(i)]
\item $\theta_{\tau}(\alpha)=\overline{0}$, $\theta_{\tau}(\beta)=\overline{1}$ and $\theta_{\tau}(a_i)=\overline{0}$   for all $i\geq 1$, which is case~(\ref{it:casec1}).
\item  $\theta_{\tau}(\alpha)=\theta_{\tau}(\beta)=\overline{0}$,  $\theta_{\tau}(a_1)=\overline{1}$ and $\theta_{\tau}(a_i)=\overline{0}$ for all $i>1$, which is case~(\ref{it:casec2}).

\item  $\theta_{\tau}(\alpha)=\overline{0}$, $\theta_{\tau}(\beta)=\overline{1}$,  $\theta_{\tau}(a_1)=\overline{1}$ and
$\theta_{\tau}(a_i)=\overline{0}$ for all $i>1$.
\item  $\theta_{\tau}(\alpha)=\overline{1}$, $\theta_{\tau}(\beta)=\overline{0}$, and  $\theta_{\tau}(a_i)=\overline{0}$   for all $i\geq 1$, which is case~(\ref{it:casec3}).
\item    $\theta_{\tau}(\alpha)=\overline{1}$, $\theta_{\tau}(\beta)=\overline{0}$,  $\theta_{\tau}(a_1)=\overline{1}$ and $\theta_{\tau}(a_i)=\overline{0}$ for all $i>1$.
\end{enumerate}
We now reduce these five cases to three.
Let  
\begin{equation}\label{eq:subsast}
\left(a^{\ast}, b^{\ast}, c^{\ast}, d^{\ast}\right)=\left(acac^{-1}a^{-1}, aca^{-1}c^{-1}bac^{-1}a^{-1},
aca^{-1}, da^{-1}\right).
\end{equation}
Then 
\begin{align*}
[a^{\ast},b^{\ast}]'[c^{\ast}, d^{\ast}]=&
acac^{-1}a^{-1} aca^{-1}c^{-1}bac^{-1}a^{-1} acac^{-1}a^{-1}
aca^{-1}b^{-1}cac^{-1}a^{-1}\ldotp\\
&   aca^{-1} da^{-1} ac^{-1}a^{-1} ad^{-1}=[a,b]'[c, d].
\end{align*}
The substitution~\reqref{subsast} shows that among the above five subcases, the second subcase is equivalent to the third, and the fourth is equivalent to
 $\theta_{\tau}(\alpha)=\overline{1}$, $\theta_{\tau}(\beta)=\overline{0}$, and  $\theta_{\tau}(a_2)=\overline{1}$ and
$\theta_{\tau}(a_i)=\overline{0}$ for $i\neq 2$. But from  the proof of  \repr{appendop},
this is equivalent to  $\theta_{\tau}(\alpha)=\overline{1}$, $\theta_{\tau}(\beta)=\overline{0}$,
and  $\theta_{\tau}(a_1)=\overline{1}$ and $\theta_{\tau}(a_i)=\overline{0}$ for $i>1$, which is the fifth subcase. This completes the proof of part~(\ref{it:casebII}), and thus that of the proposition.\qedhere
\end{enumerate}
\end{enumerate}
\end{proof}

\begin{rem} 
For each of the three cases ($h$ odd, $h=2$ and $h\geq 4$ even) listed 
above in \repr{appendor}, the corresponding subcases are not equivalent. 
To see this, let us first consider the case $h=2$. Using a set of generators 
for $\operatorname{Out}(N_2)$, it follows that the two subcases are not equivalent. For the 
case $h$ odd we use the following observations. It is a general fact that an 
automorphism of $\pi_1(N_h)$ maps orientable loops to orientable loops and 
non-orientable loops to non-orientable loops. Moreover, consider the induced 
automorphism on the Abelianisation of $\pi_1(N_h)$. Since the class of 
the generator $v$ given in the presentation of $\pi_1(N_h)$    
generates the torsion part of the Abelianisation of $\pi_1(N_h)$, the subgroup generated by 
the class of $v$ is invariant under any homomorphism.  These two facts tell us that the class 
of $v$ in the Abelianisation is mapped into itself, and that the subgroup generated 
by the classes of the elements $a_1,\ldots,a_{h-1}$ is also invariant. A straightforward analysis using these two 
properties shows that  the three subcases cannot be equivalent. 
The last case, $h\geq 4$ even, can be obtained by arguing in a similar way, and is left to the reader.
\end{rem}


\begin{thebibliography}{BGHZ}

{\small


\bibitem[BGHZ]{BGHZ} A.~Bauval, D.~L.~Gon\c{c}alves, C.~Hayat and P.~Zvengrowski, The Borsuk-Ulam type theorem for double of all Seifert manifolds, preprint.

\bibitem[Be]{Be} P.~Bellingeri, On presentations of surface braid groups, \emph{J.~Algebra}  \textbf{274}  (2004), 543--563.


\bibitem[Bo]{Bo} K.~Borsuk, Drei S\"atze  \"uber die $n$-dimensionale
Euklidische Sph\"are, \emph{Fundamenta Mathematicae} \textbf{20} (1933), 177--190.

\bibitem[Br]{Br} K.~S.~Brown, Cohomology of groups, Graduate Texts in
Mathematics, \textbf{87} Springer-Verlag, New York-Berlin (1982).



\bibitem[EL]{EL}  B.~Eckmann and P.~Linnell, Poincar\'e duality groups of
dimension two II, \emph{Comment.\ Math.\ Helvetici} \textbf{58} (1983), 111--114.

\bibitem[EM]{EM} B.~Eckmann and 
H.~M\"uller, Poincar\'e duality groups of dimension two, \emph{Comment.\ Math.\ Helvetici} \textbf{55} (1980), 510--520

\bibitem[FH]{FH} E.~Fadell and S.~Husseini, The Nielsen number on surfaces. Topological methods in nonlinear functional analysis (Toronto, Ont., 1982), 59--98, \emph{Contemp. Math.}, \textbf{21}, Amer.\ Math.\ Soc., Providence, RI, 1983. 

\bibitem[FN]{FaN} E.~Fadell and L.~Neuwirth, Configuration spaces,
\emph{Math.\ Scandinavica} \textbf{10} (1962), 111--118.



\bibitem[Go]{Go} D.~L.~Gon\c calves, The Borsuk-Ulam theorem for surfaces,
\emph{Quaestiones Mathematicae} \textbf{29} (2006), 117--123.

\bibitem[GG1]{GG1} D.~L.~Gon\c{c}alves and J.~Guaschi, On the
structure of surface pure braid groups,  \emph{J.\ Pure Appl.\ Algebra}
\textbf{182} (2003), 33--64 (due to a printer's error, this article was republished in its entirety
with the reference \textbf{186} (2004), 187--218).

\bibitem[GG2]{GG2} D.~L.~Gon\c{c}alves and J.~Guaschi, The braid
groups of the projective plane, \emph{Algebraic and Geometric Topology} 
\textbf{4} (2004), 757--780.

\bibitem[GG3]{GG4} D.~L.~Gon\c{c}alves and J.~Guaschi, The braid group
$B_{{n,m}}(\St[2])$ and the generalised Fadell-Neuwirth short exact
sequence, \emph{J.~Knot Theory and its Ramifications} \textbf{14}
(2005), 375--403.

\bibitem[GG4]{GG3} D.~L.~Gon\c{c}alves and J.~Guaschi, Braid groups of non-orientable surfaces and the Fadell-Neuwirth short exact sequence, \emph{J.\ Pure Appl.\ Algebra} \textbf{214} (2010), 667--677.

\bibitem[GHZ]{GHZ} D.~L.~Gon\c{c}alves, C.~Hayat and P.~Zvengrowski, The Borsuk-Ulam theorem for manifolds, with applications to dimensions two and three, preprint.

\bibitem[GNS]{GNS} D.~L.~Gon\c{c}alves, O.~Manzoli Neto and M.~Spreafico, 
The Borsuk-Ulam theorem for $3$-space forms, preprint (2009).



\bibitem[LS] {LS}  L.~Lusternik  and  L.~Schnirelmann, M\'ethodes topologiques dans les probl\`emes variationnels,  (Russian) \emph{Moskau: Issledowatelskij Institut Mathematiki i Mechaniki} pri JMGU (1930). 


\bibitem[Mat]{Mat}  J. Matou\v{s}ek, Using the Borsuk-Ulam Theorem, Universitext,
Springer-Verlag, Berlin, Heidelberg, New York (2002).

\bibitem[My]{My} R.~Myers, Free involutions on lens spaces, \emph{Topology} 
\textbf{20} (1981), 313--318.


\bibitem[S]{S} G.~P.~Scott, Braid groups and the group of
homeomorphisms of a surface, \emph{Proc.\  Camb.\ Phil.\ Soc.} 
\textbf{68} (1970), 605--617.

\bibitem[VB]{vB} J.~Van Buskirk, Braid groups of compact $2$-manifolds with
elements of finite order, \emph{Trans.\ Amer.\ Math.\ Soc.\/} \textbf{122} (1966), 81--97.

\bibitem[Wh]{Wh} G.~Whitehead, Elements of Homotopy Theory, Springer
Verlag, New York, 1978.


}

\end{thebibliography}
\end{document}